\documentclass[12pt]{amsart}%

\usepackage{amssymb}
\usepackage{amsmath}
\usepackage{amsthm}
\usepackage{amscd,mdwlist}
\usepackage[margin=1in]{geometry}
\usepackage{hyperref}
\usepackage{color}
\usepackage{cite}
    
\theoremstyle{plain}
\newtheorem{theorem}{Theorem}[section]
\newtheorem{proposition}{Proposition}[section]
\newtheorem{corollary}{Corollary}[section]
\newtheorem{lemma}{Lemma}[section]
\newtheorem{definition}{Definition}[section]

\theoremstyle{remark}
\newtheorem{remark}{Remark}[section]

\DeclareMathOperator{\supp}{supp}
\DeclareMathOperator{\diam}{diam}

\DeclareMathOperator{\im}{Im}

\begin{document}

\title{On the viscous Burgers equation on metric graphs and fractals}

\author{Michael Hinz$^1$, Melissa Meinert$^2$}
\thanks{$^1$, $^2$ Research supported in part by the DFG IRTG 2235: 'Searching for the regular in the irregular: Analysis of singular and random systems'.}
\thanks{$^1$ Research supported in part by the 'Fractal Geometry and Dynamics' program, Institut Mittag-Leffler, Stockholm, 2017, and by the DFG CRC 1283: 'Taming uncertainty and profiting from randomness and low regularity in analysis, stochastics and their applications'.}
\address{$^1$Fakult\"at f\"ur Mathematik, Universit\"at Bielefeld, Postfach 100131, 33501 Bielefeld, Germany}
\email{mhinz@math.uni-bielefeld.de}
\address{$^2$Fakult\"at f\"ur Mathematik, Universit\"at Bielefeld, Postfach 100131, 33501 Bielefeld, Germany}
\email{mmeinert@math.uni-bielefeld.de}

\begin{abstract}

We study a formulation of Burgers equation on the Sierpinski gasket, which is the prototype of a p.c.f. self-similar fractal. One possibility is to implement Burgers equation as a semilinear heat equation associated with the Laplacian for scalar functions, just as on the unit interval. Here we propose a second, different formulation which follows from the Cole-Hopf transform and is associated with the Laplacian for vector fields. The difference between these two equations can be understood in terms of different vertex conditions for Laplacians on metric graphs. For the second formulation we show existence and uniqueness of solutions and verify the continuous dependence on the initial condition. We also prove that solutions on the Sierpinski gasket can be approximated in a weak sense by solutions to corresponding equations on approximating metric graphs. These results are part of a larger program discussing non-linear partial differential equations on fractal spaces.
\tableofcontents
\end{abstract}
\maketitle

\section{Introduction}

In this article we investigate the viscous Burgers equation on fractals. We propose a formulation of the equation on the Sierpinski gasket, endowed with its standard energy form and an arbitrary finite Borel measure with full support. We
define a notion of solution and verify the existence and uniqueness of solutions for initial conditions that are gradients of energy finite functions. In addition, we verify the continuous dependence of the solution on the initial conditions and provide an approximation of solutions by means of solutions to Burgers equations on approximating metric graphs. Our main tool is the Cole-Hopf transform, which also dictates the way we phrase the equation.

The viscous Burgers equation, \cite{Bu40, Bu48}, is one of the simplest nonlinear partial differential equations, on the real line it reads
\begin{equation}\label{E:unitBurgers}
u_t=\sigma u_{xx}-u_xu,
\end{equation}
see for instance \cite{Evans, Olver93, Olver14}. The nonlinear term $u_xu=\frac12(u^2)_x$ models a convection effect and the viscosity parameter $\sigma>0$ determines the strength of a competing diffusion. Physically reasonable solutions of the inviscid Burgers equation (the case $\sigma=0$) can be obtained as the limit of solutions to (\ref{E:unitBurgers}) for $\sigma\to 0$. On higher dimensional Euclidean domains or on manifolds (\ref{E:unitBurgers}) translates into 
\begin{equation}\label{E:RdBurgers}
u_t=\sigma \Delta u - \left\langle u, \nabla\right\rangle u,
\end{equation}
and this equation is also seen as a simplification of the Navier-Stokes equation. Here $\Delta$ is the Laplacian acting on vector fields.  Sometimes Burgers equation is formulated with $\frac12\nabla \left\langle u,u\right\rangle$ in place of $\left\langle u, \nabla\right\rangle u$, for gradient field solutions $u=\nabla h$ the terms agree. Equation (\ref{E:RdBurgers}) can be solved using the Cole-Hopf transform, \cite{Co51, Florin48, Hopf50}: If $w$ is a positive solution to the heat equation $w_t=\sigma\Delta w$, now with the Laplacian $\Delta$ acting on scalar valued functions, then the gradient field $u:=-2\sigma\nabla \log w$ solves (\ref{E:RdBurgers}). See also \cite{Bi03}. This transform is one example of an entire hierarchy of transforms, \cite{KuSh09, Tasso76}, and naturally related to integrable systems, \cite{Olver93}.

Analysis on fractals is still a relatively young area of research, \cite{Ba, Ki01, Ki03, Ki12, Ku89, Str06}. Classical notions of differentiation are not available, the analysis is based on energy (Dirichlet) forms and diffusion processes, \cite{FOT94}. The major part of the existing literature on partial differential equations on fractals is devoted to the study of linear elliptic or linear or semilinear parabolic equations for scalar valued functions. The literature on other types of equations is rather sparse, and studies of equations involving first order differential operators have been started only recently, e.g. in \cite{HRT13}, based on the first order calculus for Dirichlet forms proposed in \cite{CS03, CS09} and studied further in \cite{H14, HKT13, HR16, HRT13, HTams, HTc, IRT12}. The implementation of such equations is nontrivial, because the Dirichlet forms involved are not immediately given as integrals involving gradient operators. In fact, the definition of an associated gradient operator is a nontrivial subsequent step, \cite{CS03, CS09, HRT13, IRT12}.

Here our main aim is to propose a formulation of Burgers equation on sufficiently simple fractals that can be solved using the Cole-Hopf transform. For notational simplicity we consider the viscosity $\sigma=1$ only. Because it is the prototype of the easy to handle class of p.c.f. self similar fractals, \cite{Ki01}, we consider the equation on the Sierpinski gasket. In terms of analysis, the Sierpinski gasket shares many properties with compact subintervals of the real line, and two conceptually different generalizations emerge. Interpreting (\ref{E:unitBurgers}) (with $\sigma=1$) as a semilinear heat equation for scalar functions motivates a formulation of Burgers equation on the Sierpinski gasket as the formal problem
\begin{equation}\label{E:scalarBurgers}
\begin{cases}
g_t(t)&=  -d^\ast d g(t) - \frac12d(g^2)(t),\\
g(0)&=g_0,
\end{cases}
\end{equation}
where we symbolically write $d$ for the gradient operator taking a function into a vector field and $d^\ast$ for its adjoint (such that $-d^\ast$ is the divergence operator). This is a semilinear heat equation for the Laplacian $-d^\ast d$ acting on functions. In \cite{LiuQian17} it has been implemented as an $L^2$-Cauchy problem with respect to the natural self-similar Hausdorff measure, and the authors showed existence, uniqueness and regularity of weak solutions for (\ref{E:scalarBurgers}) with Dirichlet boundary conditions. As discussed in \cite{LiuQian16, LiuQian17}, this model is naturally related to control theory and (backward) stochastic differential equations. However, it cannot be solved using the Cole-Hopf transform. An alternative viewpoint upon (\ref{E:unitBurgers}) is to interpret it as an equation for vector fields, similar to (\ref{E:RdBurgers}). This suggests to formulate Burgers equation on the Sierpinski gasket as the formal problem
\begin{equation}\label{E:vectorBurgers}
\begin{cases}
u_t(t)&=  -d\:d^\ast u(t) - \frac12d(u^2)(t),\\
u(0)&=u_0.
\end{cases}
\end{equation}
Here $d\:d^\ast$ is the Laplacian acting on vector fields, so that (\ref{E:vectorBurgers}) has to be seen as a vector equation. The study of this model is the objective of the present article, and as mentioned above it can be implemented using first order calculus, \cite{CS03, HRT13, IRT12}. The volume measure can be fairly general. We define $d\:d^\ast$ and $u\mapsto \frac12 d(u^2)$ in distributional sense and use the Cole-Hopf transform to verify existence and uniqueness of solutions in the case that the initial condition is a gradient of an energy finite function. We also show their continuous dependence on the initial condition. 

The difference between (\ref{E:scalarBurgers}) and (\ref{E:vectorBurgers}) admits a very natural interpretation if one considers these equations on metric graphs, \cite{BK16, BLS2009, EP07, FKW07, Hae11, KS99, KS00, Ku04, Ku05, Mugnolo14}.  In this case (\ref{E:scalarBurgers}) is a semilinear heat equation for the Laplacian $d^\ast\:d$ with Kirchhoff vertex conditions, while (\ref{E:vectorBurgers}) employs the Laplacian $d\:d^\ast$ with another, different type of vertex conditions. We also verify existence and uniqueness of solutions to (\ref{E:vectorBurgers}) on compact metric graphs, as well as continuous dependence on initial conditions. In the metric graph case the operators involved and their domains admit fairly explicit expressions.

The energy form on the Sierpinski gasket can be written as the limit of energy forms on a sequence of metric graphs approximating the gasket, see for instance \cite{T08}. This raises the question whether solutions of (\ref{E:vectorBurgers}) on the Sierpinski gasket can be approximated by solutions of (\ref{E:vectorBurgers}) on the approximating metric graphs. If so, this might be regarded as another piece of evidence that on the gasket formulation (\ref{E:vectorBurgers}) is meaningful. In order to establish such a result we again use the Cole-Hopf transform and first verify a corresponding statement for solutions of heat equations, in other words, a generalized strong resolvent convergence for the Laplacians for scalar functions on varying $L^2$-space. A suitable concept has been established in \cite{KS03}, see for instance \cite{Hi09} for an application to fractals. However, practically it seems difficult to verify the characterization of such a convergence in terms of Dirichlet forms. It is much easier to verify sufficient conditions for generalized norm resolvent convergence of self-adjoint operators as considered in \cite{P12, PS17, PS17a}. This can be done in a quite straightforward manner if one uses the concept of $\delta$-quasi unitary equivalence introduced in \cite[Chapter 4, in particular, Definition 4.4.11, Proposition 4.4.15 and Theorem 4.2.10]{P12}. A related concept for sectorial operators was provided in \cite{MNP13}. Mimicking the proof of \cite[Theorem 1.1]{PS17} (where a similar approximation along a sequence of discrete graphs was shown), we verify the
norm resolvent convergence of the Laplacians. As a consequence we obtain the convergence of solutions of the heat equations in $L^2$ in the strong sense and in the Dirichlet form domain in the weak sense. This is sufficient to verify that solutions of (\ref{E:vectorBurgers}) on approximating metric graphs converge to the solution of (\ref{E:vectorBurgers}) on the gasket in a suitable weak sense. To formulate the identification operators involved we rely on approximations by piecewise harmonic respectively edge-wise linear functions. 

We finally like to mention that a version of the viscous Burgers equation (\ref{E:unitBurgers}) on the real line appears as the limiting deterministic equation for the weakly asymmetric nearest neighbor exclusion process, \cite{Gae88, DMPS87}.
However, on the Sierpinski gasket the deterministic equations related to weakly asymmetric exclusion processes do involve nonlinearities of divergence type rather than of gradient type, see for instance \cite{ChHT17}. Further ideas, such as studies of inviscid Burgers equations and vanishing viscosity limits on suitable fractal spaces will be subject of subsequent studies.

The present paper is organized as follows. In Section \ref{S:metric} we recall basics on metric graphs, related energies and Laplacians, provide adequate formulations of (\ref{E:scalarBurgers}) and (\ref{E:vectorBurgers}) and prove existence, uniqueness and continuous dependence on initial conditions for (\ref{E:vectorBurgers}). In Section \ref{S:SG} we first recall basic concepts of the analysis on the Sierpinski gasket and the related vector analysis. We briefly discuss problem (\ref{E:scalarBurgers}) and the results of \cite{LiuQian17} and then formulate (\ref{E:vectorBurgers}). Again we state existence, uniqueness and continuous dependence on initial conditions, the proofs are only minor modifications compared to the metric graph case. Section \ref{S:approx} provides metric graph approximation results for solutions of the heat equation and for solutions of (\ref{E:vectorBurgers}). An appendix contains some technical calculations and, to make the paper self-contained, a proof of the generalized norm resolvent convergence.

For quantities  $(f,g)\mapsto \mathcal{Q}(f,g)$ depending on two arguments $f,g$ in a symmetric way we use the notation $\mathcal{Q}(f):=\mathcal{Q}(f,f)$.


\section*{Acknowledgements} We would like to thank Xuan Liu, Olaf Post, Robert Schippa, Jan Simmer and Alexander Teplyaev for inspiring and useful comments.

\section{Heat and Burgers equations on metric graphs}\label{S:metric}

\subsection{Preliminaries on metric graphs} Mainly following \cite{BLS2009, Hae11} we provide some basics on metric graphs. A \emph{metric graph} is a quadruple $\Gamma = (E, V, i, j)$ consisting of a countable set $E$ of open intervals $e=(0,l_e)$ with $l_e\in (0,+\infty]$, a countable set $V$ and maps $i: E \rightarrow V$ and $ j:\{e \in E \mid l_e <+\infty \} \rightarrow V $. To the elements $e$ of $E$ we refer as \emph{edges}, to the elements $v$ of $V$ as \emph{vertices}. Given $e=(0,l_e)\in E$, we call $l_e$ the \emph{length} of $e$, $i(e)$ its \emph{initial} and $j(e)$ its \emph{terminal vertex}. An edge $e\in E$ and a vertex $p\in V$ are said to be \emph{incident}, $e\sim p$, if $p$ is the initial or the terminal vertex of $e$. Two distinct vertices $p,q\in V$ are said to be \emph{neighbors}, $p\sim q$, if they are incident to the same edge.
The graph $\Gamma$ is said to have \emph{no loops} if there is no edge $e$ with $i(e)=j(e)$ and that $\Gamma$ has \emph{no multiple edges} if for any two different edges $e$ and $e'$ the sets $\{i(e), j(e)\}$ and $\{i(e'),j(e')\}$ are different. A metric graph $\Gamma$ is called \emph{connected}, if for any distinct $p,q\in V$ there exists $p_0,...,p_n\in V$ such that $p_0=p$, $p_n=q$ and $p_i\sim p_{i-1}$ for $i=1,..,n$. We set $X_e := \{e\} \times (0,l_e)$ and define the disjoint union 
\begin{equation}\label{E:spaceX}
X_\Gamma:= V \cup \bigcup_{e \in E} X_e. 
\end{equation}
For any edge $e$ let $\pi_e : X_e \rightarrow (0, l_e)$ denote the projection $(e,t) \mapsto t$ onto the second component of $X_e$. For $e\in E$ with $l_e<+\infty$ we set $\bar X_e := X_e \cup \{ i(e), j(e) \}$ and for $e\in E$ with $l_e=+\infty$ we set 
$\bar X_e := X_e \cup \{ i(e)\}$. Let $X_\Gamma$ be endowed with the unique topology such that for any $e\in E$ the mapping $\pi_e$ 
extends to a homeomorphism $\pi_e : \bar X_e \rightarrow [0, l(e)]$ that satisfies $\pi_e(i(e))=0$ and, in case that $l_e<+\infty$, also $\pi_e(j(e))=l(e)$. Given a real valued function $f$ on $X_\Gamma$ we define a function on each edge $e\in E$ by $f_e:=f\circ\pi_e^{-1}$. If $f$ is continuous on $X_\Gamma$ then for each $e\in E$ the function $f_e$ is continuous on $e$ and its value at each vertex is the limit of its values on any adjacent edge. Moreover, the canonical length metric metrizes this topology and makes $X$ into a locally compact separable metric space.  The space $X_\Gamma$ is compact if and only if $E$ is a finite set and all edges have finite length, and $\Gamma$ is called \emph{compact} if $X_\Gamma$ is compact. In what follows we assume that \emph{$\Gamma$ is a compact connected metric graph having no loops or multiple edges.}

On each edge $e\in E$ let $\dot{W}^{1,2}(e)$ denote the homogeneous Sobolev space consisting of locally Lebesgue integrable functions $g$ on $e$ such that 
\[\mathcal{E}_e(g):=\int_0^{l_e} (g'(s))^2\:ds<+\infty,\]
where the derivative $g'$ of $g$ is understood in the distributional sense.

For a function $f$ on $X$ such that $f_e\in  \dot{W}^{1,2}(e)$ for any $e\in E$ we can define its energy $\mathcal{E}_\Gamma(f)$ on $\Gamma$ by the sum 
\[\mathcal{E}_\Gamma(f):=\sum_{e\in E} \mathcal{E}_e (f_e).\]
The space of continuous functions on $\Gamma$ with finite energy we denote by
\[\dot{W}^{1,2}(X_\Gamma):=\{f\in C(X_\Gamma): \text{for any $e\in E$ we have $f_e\in  \dot{W}^{1,2}(e)$, and $\mathcal{E}_\Gamma(f)<+\infty$}\}.\]
By polarization we obtain a nonnegative definite symmetric bilinear form $(\mathcal{E}_{\Gamma},\dot{W}^{1,2}(X_\Gamma))$ satisfying the Markov property. Moreover, $(\mathcal{E}_{\Gamma},\dot{W}^{1,2}(X_\Gamma))$ is a resistance form on $X$ in the sense of \cite[Definition 2.8]{Ki03}. In particular, on any single edge $e\in E$ the form $\mathcal{E}_e$ satisfies 
\begin{equation}\label{E:resistonedge}
(f_e(s)-f_e(s'))^2\leq l_e\mathcal{E}_{e}(f_e)
\end{equation}
for any $f\in \dot{W}^{1,2}(X_\Gamma)$ and any $s,s'\in e$. 

Now suppose $\mu_\Gamma$ is an atom free nonnegative Radon measure on $X_\Gamma$ with full support. Then $(\mathcal{E}_\Gamma,\dot{W}^{1,2}(X_\Gamma))$ is a strongly local regular Dirichlet form on $L^2(X_\Gamma,\mu_\Gamma)$ in the sense of \cite{FOT94}. We write $W^{1,2}(X_\Gamma,\mu_\Gamma)$ for the Hilbert space $\dot{W}^{1,2}(X_\Gamma)$ with norm 
\begin{equation}\label{E:W12norm}
\left\|f\right\|_{W^{1,2}(X_\Gamma,\mu_\Gamma)}:=\left(\mathcal{E}_\Gamma(f)+\left\|f\right\|_{L^2(X_\Gamma,\mu_\Gamma)}^2\right)^{1/2}, \quad f\in W^{1,2}(X_\Gamma,\mu_\Gamma).
\end{equation}
A function $f\in W^{1,2}(X_\Gamma,\mu_\Gamma)$ has zero energy $\mathcal{E}_\Gamma(f)=0$ if and only if $f$ is constant on $X_\Gamma$, and 
\begin{equation}\label{E:Linftybound}
\left\|f\right\|_{\sup}\leq c\:\left\|f\right\|_{W^{1,2}(X_\Gamma,\mu_\Gamma)},\quad f\in W^{1,2}(X_\Gamma,\mu_\Gamma),
\end{equation}
where $c>0$ is a constant not depending on $f$, see \cite[Corollary 2.2]{Hae11}. Alternatively, one can follow the arguments of \cite[Lemma 5.2.8]{Ki01}.

In what follows we assume $(c_e)_{e\in E}$ is a family of real numbers $c_e$ such that $\inf_{e\in E} c_e>0$ and $\sup_{e\in E} c_e<+\infty$ and that $\mu_\Gamma$ is the measure on $X_\Gamma$ determined by 
\begin{equation}\label{E:measure}
\mu_\Gamma|_{X_e}\circ \pi_e^{-1}=c_e\lambda^1|_{e},\quad e\in E,
\end{equation}
where $\lambda^1$ denotes the Lebesgue measure on the real line. This class of measures is sufficiently large for our purposes.

\subsection{Kirchhoff Laplacian} Under the stated assumption the generator of the Dirichlet form $(\mathcal{E}_{\Gamma}, W^{1,2}(X_\Gamma,\mu_\Gamma))$ is the nonpositive definite self-adjoint operator $(\mathcal{L}_\Gamma,\mathcal{D}(\mathcal{L}_\Gamma))$ on $L^2(X_\Gamma,\mu_\Gamma)$, where $\mathcal{D}(\mathcal{L}_\Gamma)$ is the collection of all $f \in W^{1,2}(X_\Gamma,\mu_\Gamma)$ such that $f_e\in W^{2,2}(e)$ for all $e\in E$ and $\sum_{e \sim p} U_p(e)f_e'(p) = 0$ for all $p \in V$ and 
\begin{equation}\label{E:metricLaplace}
\mathcal{L}f=\sum_{e\in E} c_e^{-1}\mathbf{1}_e\:f_e''
\end{equation}
for all $f\in\mathcal{D}(\mathcal{L}_\Gamma)$.
Here $f_e'(p)$ denotes the trace of $f_e'\in W^{1,2}(e)$ on $p$, and $U_p(e)=1$ if $p=j(e)$ and $U_p(e)=-1$ if $p=i(e)$, so that at both points we consider the normals outgoing from the edge e (and ingoing into $i(e)$ and $j(e)$, respectively). To $(\mathcal{L}_\Gamma,\mathcal{D}(\mathcal{L}_\Gamma))$ one refers as \emph{Laplacian with Kirchhoff vertex conditions}, see e.g. \cite[Definition 5]{FKW07}. On vertices that are incident to one edge only, this forces zero Neumann boundary conditions.

A function $f\in L^2(X_\Gamma,\mu_\Gamma)$ is already uniquely determined by the functions $f_e$, and we may write $f=(f_e)_{e\in E}$. Given a function $f\in W^{1,2}(X_\Gamma,\mu_\Gamma)$ we can define a function $df=((df)_e)_{e\in E}$ in $L^2(X_\Gamma,\mu_\Gamma)$ by $(df)_e=c_e^{-1/2}f_e'$ for any $e\in E$ where each $f_e'$ is understood in distributional sense. This yields a bounded linear operator $d:W^{1,2}(X_\Gamma,\mu_\Gamma)\to L^2(X_\Gamma,\mu_\Gamma)$, note that for any $f,g\in W^{1,2}(X_\Gamma,\mu_\Gamma)$ we have
\begin{equation}\label{E:Kirchoffenergy}
\left\langle df,dg\right\rangle_{L^2(X_\Gamma,\mu_\Gamma)}=\mathcal{E}(f,g).
\end{equation}

\begin{remark}\label{R:isoiso}
Since $X_\Gamma$ is compact, the kernel of $d$ consists exactly of the constants, $\ker d=\mathbb{R}$, so that 
$W^{1,2}(X_\Gamma,\mu_\Gamma)/\mathbb{R}$ and the image $\im d$ of $d$ in $L^2(X_\Gamma,\mu_\Gamma)$ are isomorphic as vector spaces and by (\ref{E:Kirchoffenergy}) even isometrically isomorphic as Hilbert spaces.
\end{remark}

Since due to (\ref{E:Linftybound}) the space $W^{1,2}(X_\Gamma,\mu_\Gamma)$ is an algebra with pointwise multiplication, we can observe the Leibniz rule $d(fg)=(df)g+fdg$, for any $f,g$ from this space. The operator $d$ may also be seen as a densely defined closed linear operator on $L^2(X_\Gamma,\mu_\Gamma)$ with domain $W^{1,2}(X_\Gamma,\mu_\Gamma)$, and using integration by parts on the individual edges and Fubini's theorem, the adjoint $d^\ast$ of $d$ is seen to be $d^\ast f=((d^\ast f)_e)_{e\in E}$ with $(d^\ast f)_e=-c_e^{-1/2}f_e'$ for $f$ from its domain $\mathcal{D}(d^\ast)$ consisting of all $f\in L^2(X_\Gamma,\mu_\Gamma)$ such that $f_e\in \dot{W}^{1,2}(e)$ for all $e\in E$ and 
\begin{equation}\label{E:vertexdast}
\sum_{e\sim p} c_e^{1/2}U_p(e)f_e(p)=0
\end{equation} 
for all $p\in V$. Similarly as before $f_e(p)$ is understood in the sense of traces. By general theory $d^\ast$ is closed in $L^2(X_\Gamma,\mu_\Gamma)$ and its domain $\mathcal{D}(d^\ast)$ is dense. 
A function $f\in W^{1,2}(X_\Gamma,\mu_\Gamma)$ is in $\mathcal{D}(\mathcal{L}_\Gamma)$ if and only if $df$ is in $\mathcal{D}(d^\ast)$, and in this case we have $\mathcal{L}_\Gamma f=-d^\ast df$. 

\subsection{Vector Laplacian}
Viewed as the target space of the derivation $d$, the space $L^2(X_\Gamma,\mu_\Gamma)$ can also be interpreted as the space of $L^2$-vector fields. Its subspace $\ker d^\ast$ is trivial if and only if $\Gamma$ has no cycles (i.e. is a tree), see \cite[Proposition 5.1]{IRT12}. We follow \cite{BK16} and define a natural nonnegative definite closed quadratic form on the space $L^2(X_\Gamma,\mu_\Gamma)$ of $L^2$-vector fields by setting $\mathcal{D}(\vec{\mathcal{E}}_\Gamma):=\mathcal{D}(d^\ast)$ and
\begin{equation}\label{E:antiKirchhoffenergy}
\vec{\mathcal{E}}_\Gamma(u,v):=\left\langle d^\ast u, d^\ast v\right\rangle_{L^2(X_\Gamma,\mu_\Gamma)},\quad u,v\in \mathcal{D}(\vec{\mathcal{E}}_\Gamma).
\end{equation}

\begin{remark}\mbox{}
\begin{enumerate}
\item[(i)] If $\Gamma$ has only one single edge $e$, then $(\vec{\mathcal{E}}_\Gamma, \mathcal{D}(\vec{\mathcal{E}}_\Gamma))$ is the Dirichlet form associated with the Laplacian on $e$ with Dirichlet boundary conditions, \cite[Example 4.1]{BK16}.
\item[(ii)] In general $(\vec{\mathcal{E}}_\Gamma, \mathcal{D}(\vec{\mathcal{E}}_\Gamma))$ is not a Dirichlet form. Suppose $\Gamma$ has a vertex $p\in V$ with at least three incident edges $e_1$, $e_2$, $e_3$, and $c_{e_i}=1$, $i=1,2,3$. If $e_1$ and $e_2$ have $p$ as terminal and $e_3$ has it as initial vertex, consider a function $v\in \mathcal{D}(\vec{\mathcal{E}}_\Gamma)\cap L^\infty(X_\Gamma,\mu_\Gamma)$ that satisfies $v_{e_1}=1$, $v_{e_2}=1$ and $v_{e_3}=2$.
Then the square $v^2$ of $v$ violates (\ref{E:vertexdast}) at $p$. Consequently, the Markov property cannot hold.
\end{enumerate}

\end{remark}
The generator of $(\vec{\mathcal{E}}_\Gamma,\mathcal{D}(\vec{\mathcal{E}}_\Gamma))$ is the nonnegative definite self-adjoint operator $(\vec{\mathcal{L}}_\Gamma,\mathcal{D}(\vec{\mathcal{L}}_\Gamma))$, given by $\vec{\mathcal{L}}_\Gamma v:=-dd^\ast v$ for all functions $v$ from its domain $\mathcal{D}(\vec{\mathcal{L}}_\Gamma)$, which is the space of all  $v\in \mathcal{D}(d^\ast)$ such that $d^\ast v=(-v_e')_{e\in E}$ is in $W^{1,2}(X_\Gamma,\mu_\Gamma)$. 

\begin{remark}
To the vertex conditions associated with $\vec{\mathcal{L}}_\Gamma$ the authors of \cite{BK16} referred to as anti-Kirchhoff conditions, they slightly differ from those specified in \cite[Definition 6]{FKW07}.
\end{remark}

Since $X_\Gamma$ is compact, a function $f\in W^{1,2}(X_\Gamma,\mu_\Gamma)$ satisfies $\left\langle d^\ast v,f \right\rangle_{L^2(X_\Gamma,\mu_\Gamma)}=0$ for all $d^\ast v$ with $v\in \mathcal{D}(\vec{\mathcal{L}}_\Gamma)$ if and only if $f$ is constant on $X_\Gamma$: In fact, this is equivalent to requiring $\left\langle v,df\right\rangle_{L^2(X_\Gamma,\mu_\Gamma)}=0$ for such $v$, and since $\mathcal{D}(\vec{\mathcal{L}}_\Gamma)$ is dense in $L^2(X_\Gamma,\mu_\Gamma)$ this is equivalent to $f\in \ker d$. Moreover, because the constants form a closed subspace of $L^2(X_\Gamma,\mu_\Gamma)$ it follows that each function $\varphi\in W^{1,2}(X_\Gamma,\mu_\Gamma)$ can uniquely be written as a sum
\begin{equation}\label{E:decomptestfcts}
\varphi=d^\ast v + c
\end{equation}
for some $v\in\mathcal{D}(\vec{\mathcal{L}}_\Gamma)$ and $c\in\mathbb{R}$.

\subsection{Distributional definitions}

Let $(W^{1,2}(X_\Gamma,\mu_\Gamma))^\ast$ denote the dual of $W^{1,2}(X_\Gamma,\mu_\Gamma)$. We can interpret $d^\ast$ and $\mathcal{L}_\Gamma$ in the distributional sense as bounded linear operators $d^\ast: L^2(X_\Gamma,\mu_\Gamma)\to (W^{1,2}(X_\Gamma,\mu_\Gamma))^\ast$, defined by 
\begin{equation}\label{E:adjointdist}
d^\ast v\:(\varphi):=\left\langle v, d\varphi\right\rangle_{L^2(X_\Gamma,\mu_\Gamma)}, \quad \varphi \in W^{1,2}(X_\Gamma,\mu_\Gamma),
\end{equation}
and $\mathcal{L}_\Gamma:W^{1,2}(X_\Gamma,\mu_\Gamma)\to (W^{1,2}(X_\Gamma,\mu_\Gamma))^\ast$, defined by 
\begin{equation}\label{E:distL}
\mathcal{L}_\Gamma f (\varphi):=-\mathcal{E}_\Gamma(f,\varphi),\quad \varphi \in W^{1,2}(X_\Gamma,\mu_\Gamma).
\end{equation}
The operator $\vec{\mathcal{L}}_\Gamma$ may also be interpreted in the distributional sense as a bounded linear operator $\vec{\mathcal{L}}_\Gamma:L^2(X_\Gamma,\mu_\Gamma)\to (\mathcal{D}(\vec{\mathcal{L}}_\Gamma))^\ast$ defined by
\begin{equation}\label{E:vecLdist}
\vec{\mathcal{L}}_\Gamma v (w):=-d^\ast v (d^\ast w),\quad w\in\mathcal{D}(\vec{\mathcal{L}}_\Gamma),
\end{equation} 
where (\ref{E:adjointdist}) is used. Finally, we also define the operator $d$ on $L^1(X_\Gamma,\mu_\Gamma)$ in a suitable distributional sense: Let $\mathcal{D}(\vec{\mathcal{L}}_\Gamma)$  be endowed with the norm $v\mapsto \left\|d^\ast v\right\|_{W^{1,2}(X_\Gamma,\mu_\Gamma)}$ and let $(\mathcal{D}(\vec{\mathcal{L}}_\Gamma))^\ast$ denote its topological dual. We define $d:L^1(X_\Gamma,\mu_\Gamma)\to (\mathcal{D}(\vec{\mathcal{L}}_\Gamma))^\ast$  by
\begin{equation}\label{E:ddist}
df(v):=\int_{X_\Gamma} d^\ast v\:f\:d\mu_\Gamma, \quad v\in \mathcal{D}(\vec{\mathcal{L}}_\Gamma).
\end{equation}
Then $|df(v)|\leq c\left\|d^\ast v\right\|_{W^{1,2}(X_\Gamma,\mu_\Gamma)}\left\|f\right\|_{L^1(X_\Gamma,\mu_\Gamma)}$ for any $f\in L^1(X_\Gamma,\mu_\Gamma)$ by (\ref{E:Linftybound}), and for $f\in W^{1,2}(X_\Gamma,\mu_\Gamma)$ we have $df(v)=\left\langle v,df\right\rangle_{L^2(X_\Gamma,\mu_\Gamma)}$. If $\Gamma$ has a single edge $e$ only and $f_e'$ denotes the distributional derivative of $f_e$ on $e$, then $df(v)=c_e^{1/2}f_e'(v)$.

\subsection{Kirchhoff Burgers equation} On the unit interval the viscous Burgers equation is given by (\ref{E:unitBurgers}). If we now consider Burgers equation with respect to Dirichlet boundary conditions, existence and uniqueness for arbitrary finite time horizons can for instance be obtained in a monotone operator setup, \cite[Theorem 1.1 and Example 3.2]{Liu11}. If endowed with Neumann boundary conditions, the unit interval $[0,1]$ can be seen as the metric graph having only the single edge $e=(0,1)$ and vertex set $V=\left\lbrace i(e), j(e)\right\rbrace$, and this suggests to generalize the Cauchy problem for (\ref{E:unitBurgers}) to a compact connected metric graph $\Gamma$ having no loops or multiple edges by considering the formal problem
(\ref{E:scalarBurgers}). There are various ways to formulate (\ref{E:scalarBurgers}) rigorously as a Cauchy problem
\begin{equation}\label{E:scalarCauchy}
\begin{cases}
g_t(t)&=  \mathcal{L}_\Gamma g(t)-\frac12 d(g^2)(t),\\
g(0)&=g_0
\end{cases}
\end{equation}
\emph{with initial condition $g_0\in L^2(X_\Gamma,\mu_\Gamma)$}. Imposing additional Dirichlet boundary conditions on a finite subset of $\Gamma$ and assuming $g_0\in W^{1,2}(X_\Gamma,\mu_\Gamma)$, one can invoke well known semigroup methods to obtain solutions to (\ref{E:scalarCauchy}) on $\Gamma$ for sufficiently small $T$, \cite[Section 6.3, Theorem 3.1]{P83}. We strongly believe that the arguments of \cite{LiuQian17}, which make heavy use of (\ref{E:Linftybound}), can be combined with known heat kernel estimates, see \cite{Hae11} and the references cited there, to obtain global weak solutions under Dirichlet boundary conditions, \cite[Definition 4.13]{LiuQian17}.

\subsection{Burgers equation via Cole-Hopf} As before we assume $\Gamma$ is a compact connected metric graph $\Gamma$ having no loops or multiple edges. An alternative generalization of (\ref{E:unitBurgers}) to $\Gamma$ can be obtained applying the Cole-Hopf transform to solutions of the heat equation
\begin{equation}\label{E:Kirchhoffheat}
\begin{cases}
w_t(t)&=\mathcal{L}_\Gamma w(t),\quad t>0,\\
w(0)&=w_0,
\end{cases}
\end{equation}
for the Kirchhoff Laplacian $\mathcal{L}_\Gamma$. This leads to the formal problem (\ref{E:vectorBurgers}) which we discuss in the present article and which in general is different from (\ref{E:scalarBurgers}). Assume that $w_0\in L^2(X_\Gamma,\mu_\Gamma)$ is nonnegative $\mu_\Gamma$-a.e. and strictly positive on some set of positive measure $\mu_\Gamma$. The unique solution to (\ref{E:Kirchhoffheat}), seen as a Cauchy problem in $L^2(X_\Gamma,\mu_\Gamma)$, is $w(t)=e^{t\mathcal{L}_\Gamma} w_0$. For any $t>0$ the function $w(t)$ is in $W^{1,2}(X_\Gamma,\mu_\Gamma))$, it is bounded, continuous and also strictly positive on $X_\Gamma$, because $(e^{t\mathcal{L}_\Gamma})_{t>0}$ is conservative. Therefore, by the chain rule (with respect to $t$),
\begin{equation}\label{E:defh}
h:=-2\log w 
\end{equation}
defines a differentiable function $h:(0,\infty)\to W^{1,2}(X_\Gamma,\mu_\Gamma)$, and it satisfies the \emph{potential Burgers equation}
\begin{equation}\label{E:metricKPZ}
\left\langle h_t(t),\varphi\right\rangle_{L^2(X_\Gamma,\mu_\Gamma)} = \mathcal{L}_\Gamma h(t)(\varphi)  - \frac12\left\langle dh(t), dh(t)\right\rangle(\varphi)
\end{equation}
for any $\varphi\in W^{1,2}(X_\Gamma,\mu_\Gamma)$, where we write $\left\langle dh(t), dh(t)\right\rangle(\varphi):=\left\langle \varphi\: dh(t), dh(t)\right\rangle_{L^2(X_\Gamma,\mu_\Gamma)}$.
See for instance \cite[Section 8.4]{Olver14}. Its derivative
\begin{equation}\label{E:ColeHopfsol}
u(t):=d h(t),
\end{equation} 
is a function $u:(0,\infty)\to L^2(X_\Gamma,\mu_\Gamma)$, and writing $u_t(t)(v):=\left\langle u_t(t),v\right\rangle_{L^2(X_\Gamma,\mu_\Gamma)}$, $v\in \mathcal{D}(\vec{\mathcal{L}}_\Gamma)$, we can formulate (\ref{E:vectorBurgers}) rigorously as the Cauchy problem 
\begin{equation}\label{E:vectorCauchy}
\begin{cases}
u_t(t)& =\vec{\mathcal{L}}_\Gamma u(t) -\frac12 d(u^2)(t),\quad t>0,\\
u(0)&=u_0,\end{cases}
\end{equation}
where $\vec{\mathcal{L}}_\Gamma u(t):=\vec{\mathcal{L}}_\Gamma (u(t))$ and $d(u^2)(t):=d(u^2(t))$ are understood in terms of the distributional definitions (\ref{E:vecLdist}) and (\ref{E:ddist}). We use the following notion of solution.

\begin{definition}\label{D:solmetric}
A function $u\in C([0,+\infty), L^2(X_\Gamma,\mu_\Gamma))\cap C^1((0,+\infty), L^2(X_\Gamma,\mu_\Gamma))$ is called  a \emph{solution} to (\ref{E:vectorCauchy}) \emph{with initial condition $u_0\in L^2(X_\Gamma,\mu_\Gamma)$} if $u$ satisfies the first identity in (\ref{E:vectorCauchy}) in $(\mathcal{D}(\vec{\mathcal{L}}_\Gamma))^\ast$ and the second in $L^2(X_\Gamma,\mu_\Gamma)$.
\end{definition}

We first observe the structure of solutions. The space $\im d$ is a closed subspace of $L^2(X_\Gamma,\mu_\Gamma)$ and 
$L^2(X_\Gamma,\mu_\Gamma)$ admits the orthogonal decomposition $L^2(X_\Gamma,\mu_\Gamma)=\im d\oplus \ker d^\ast$.

\begin{theorem}\label{T:metricstructure}
Suppose $u$ is a solution to (\ref{E:vectorCauchy}) with initial condition $u_0$. Let $\eta_0\in \ker d^\ast$ and $h_0\in W^{1,2}(X_\Gamma,\mu_\Gamma)$ be such that $u_0=dh_0+\eta_0$. Then $u$ is of form $u(t)=dh(t)+\eta_0$, $t\geq 0$, with a function $h:[0,+\infty)\to W^{1,2}(X_\Gamma,\mu_\Gamma)$ satisfying $dh(0)=dh_0$.
\end{theorem}

\begin{proof}
For any $t\geq 0$ there exist $\eta(t)\in \ker d^\ast$, uniquely determined, and $h(t)\in W^{1,2}(X_\Gamma,\mu_\Gamma)$, unique up to an additive constant, such that $u(t)=dh(t)+\eta(t)$. Since by definition $t\mapsto u(t)$ is differentiable on $(0,+\infty)$ and continuous on $[0,+\infty)$, so is its orthogonal projection $\eta$ to $\ker d^\ast$ and therefore also $dh$, and by Remark \ref{R:isoiso} even $h$, seen as a function with values in $W^{1,2}(X_\Gamma,\mu_\Gamma)/\mathbb{R}$. For any $v\in \ker d^\ast\subset \mathcal{D}(\vec{\mathcal{L}}_\Gamma)$ and any fixed $t$ we have 
\[\left\langle \eta_t(t),v\right\rangle_{L^2(X_\Gamma,\mu_\Gamma)}=-\left\langle h_t(t),d^\ast v\right\rangle_{L^2(X_\Gamma,\mu_\Gamma)}-d^\ast u(t)(d^\ast v)-\frac12\int_{X_\Gamma} u^2(t)\:d^\ast v\:d\mu_\Gamma=0,\]
so that also $\left\langle \eta(t)-\eta_0,v\right\rangle_{L^2(X_\Gamma,\mu_\Gamma)}=\int_0^t\left\langle \eta_\tau(\tau),v\right\rangle_{L^2(X_\Gamma,\mu_\Gamma)}d\tau=0$. However, this implies that $\eta(t)-\eta_0\:\bot\: \ker d^\ast$, which means this difference must be zero in $L^2(X_\Gamma,\mu_\Gamma)$.
\end{proof}

The Cole-Hopf transform (\ref{E:defh}) and (\ref{E:ColeHopfsol}) guarantees the existence and uniqueness of solution fields for initial conditions of gradient type.
\begin{theorem}\label{T:exmetric}
Assume that $u_0=dh_0$ with $h_0\in W^{1,2}(X_\Gamma,\mu_\Gamma)$. Let $w(t)$ denote the unique solution $e^{t\mathcal{L}_\Gamma}w_0$ to (\ref{E:Kirchhoffheat}) with initial condition $w_0:=e^{-h_0/2}$. Then the function $u(t):=-2d\log w(t)$, $t\geq 0$,
is the unique solution to (\ref{E:vectorCauchy}). 
\end{theorem}

Both the existence and the uniqueness part follow well-known standard arguments, see for instance \cite{Bi03} or \cite[Section 8.4]{Olver14}. We adapt them to our setup.

\begin{proof} To verify that $u$ is a solution, let $h$ be as in (\ref{E:defh}). The stated hypotheses imply $u_t(t)=d h_t(t)$ in $L^2(X_\Gamma,\mu_\Gamma)$ for any $t>0$. We have $\left\langle d^\ast u(t), d^\ast v\right\rangle =\left\langle \mathcal{L}_\Gamma h(t), d^\ast v \right\rangle$ for test functions $\varphi=d^\ast v$ with $v\in\mathcal{D}(\vec{\mathcal{L}}_\Gamma)$. From (\ref{E:metricKPZ}) it follows that $u$ satisfies the first identity in (\ref{E:vectorCauchy}). To verify the continuity of $u$ at zero, note that by nonnegativity and conservativity of the semigroup we have $\inf_{s\in X_\Gamma} w(t,s)\geq e^{- \left\|h_0\right\|_{\sup}  /2}$ for any $t\geq 0$. Since the function $w:[0,+\infty)\to W^{1,2}(X_\Gamma,\mu_\Gamma)$ is continuous, also its reciprocal $w(\cdot)^{-1}:[0,+\infty)\to W^{1,2}(X_\Gamma,\mu_\Gamma)$ is continuous. Therefore 
\begin{align}
\left\|u(t)-u_0\right\|_{L^2(X_\Gamma,\mu_\Gamma)}&\leq 2\left\|(dw(t)-dw_0)w(t)^{-1}\right\|_{L^2(X_\Gamma,\mu_\Gamma)}+2\left\|(w(t)^{-1}-w_0^{-1})dw_0\right\|_{L^2(X_\Gamma,\mu_\Gamma)}\notag\\
&\leq 2e^{\left\|h_0\right\|_{\sup}   /2}\left\|w(t)-w_0\right\|_{W^{1,2}(X_\Gamma,\mu_\Gamma)}+2e^{\left\|h_0\right\|_{\sup}  /2}\left\|w(t)^{-1}-w_0^{-1}\right\|_{L^\infty(X_\Gamma,\mu_\Gamma)},\notag
\end{align}
what by (\ref{E:Linftybound}) converges to zero as $t$ goes to zero. 

To see uniqueness we may, by Theorem \ref{T:metricstructure}, assume that  $u(t)\in \im d$ for any $t\geq 0$. In this case there is a potential $\tilde{h}:[0,+\infty)\to W^{1,2}(X_\Gamma,\mu_\Gamma)$ such that $u(t)=d\tilde{h}(t)$ for all $t\geq 0$. According to Definition \ref{D:solmetric} $\tilde{h}$ is continuous on $[0,+\infty)$ and differentiable on $(0,+\infty)$, and we have (\ref{E:metricKPZ}) for $\tilde{h}$ in place of $h$ and  all test functions $\varphi$ of type $\varphi=d^\ast v$, $v\in\mathcal{D}(\vec{\mathcal{L}}_\Gamma)$. In order to have (\ref{E:metricKPZ}) for all test functions from $W^{1,2}(X_\Gamma,\mu_\Gamma)$, which also detect additive constants, we need to readjust the choice of the potential. For each $t\geq 0$ set now
\[g(t):=\frac{1}{\mu_{\Gamma}(X_\Gamma)}\left\lbrace -\big\langle \tilde{h}_t(t),\mathbf{1}\big\rangle_{L^2(X_\Gamma,\mu_\Gamma)}-\frac12 \big\langle d\tilde{h}(t),d\tilde{h}(t)\big\rangle_{L^2(X_\Gamma,\mu_\Gamma)}\right\rbrace\]
and let $G:[0,+\infty)\to\mathbb{R}$ be a differentiably function satisfying $G_t=g$. Then the readjusted potential $h(t):=\widetilde{h}(t)+G(t)$
satisfies (\ref{E:metricKPZ}) for all $\varphi\in W^{1,2}(X_\Gamma,\mu_\Gamma)$: Suppose the decomposition (\ref{E:decomptestfcts}) of $\varphi$ reads $\varphi=d^\ast v+c$ with $v\in\mathcal{D}(\vec{\mathcal{L}})$ and $c\in\mathbb{R}$, then
\begin{align}
\left\langle h_t(t),d^\ast v+c\right\rangle_{L^2(X_\Gamma,\mu_\Gamma)}&=\big\langle \widetilde{h}_t(t),d^\ast v+c\big\rangle_{L^2(X_\Gamma,\mu_\Gamma)}+\big\langle\widetilde{h}_t(t),c\big\rangle_{L^2(X_\Gamma,\mu_\Gamma)}+cg(t)\mu_\Gamma(X_\Gamma)\notag\\
&=\mathcal{L}\widetilde{h}(t)(d^\ast v)-\frac12\big\langle (d^\ast v+c)d\widetilde{h}(t),d\widetilde{h}(t)\big\rangle_{L^2(X_\Gamma,\mu_\Gamma)} \notag\\
&=\mathcal{L}_\Gamma h(t)(d^\ast v+c)-\frac12\big\langle (d^\ast v+c)dh(t),dh(t)\big\rangle_{L^2(X_\Gamma,\mu_\Gamma)},\notag
\end{align}
where we have used (\ref{E:distL}) and $\ker d=\mathbb{R}$. As a consequence, the continuous $W^{1,2}(X_\Gamma,\mu_\Gamma)$-valued  function 
\[w(t):=e^{-h(t)/2},\quad t\geq 0,\]
is the unique solution to the Cauchy problem for the heat equation (\ref{E:Kirchhoffheat}) in $L^2(X_\Gamma,\mu_\Gamma)$ with initial condition $w_0$. To see this, note that 
\begin{multline}
\mathcal{L}_\Gamma w(t)=d^\ast \left(-\frac12\:e^{-h(t)/2}dh(t)\right)\notag\\
=-\frac12\:e^{-h(t)/2}\left(\mathcal{L}_\Gamma h(t)-\frac12\left\langle dh(t),dh(t)\right\rangle \right)=-\frac12\:e^{-h(t)/2}h_t(t) =w_t(t)\notag
\end{multline}
in $(W^{1,2}(X_\Gamma,\mu_\Gamma))^\ast$, which follows from \cite[Lemma 3.2]{HRT13}. However, since $w_t(t)$ is in $L^2(X_\Gamma,\mu_\Gamma)$, also $\mathcal{L}_\Gamma w(t)$ must be in $L^2(X_\Gamma,\mu_\Gamma)$, and since $W^{1,2}(X_\Gamma,\mu_\Gamma)$ is dense in $L^2(X_\Gamma,\mu_\Gamma)$ the equality must hold in $L^2(X_\Gamma,\mu_\Gamma)$. If now $\overline{u}$ was another solution of (\ref{E:vectorCauchy}) with initial condition $u_0$ different from $u$ and having a potential $\overline{h}$ satisfying (\ref{E:metricKPZ}) for all $\varphi\in W^{1,2}(X_\Gamma,\mu_\Gamma)$ then $h$ and $\overline{h}$ would have to differ on $(0,+\infty)$ by a nonconstant function. However, this would lead to two different solutions $w$ and $\overline{w}$ of the Cauchy problem (\ref{E:Kirchhoffheat}), a contradiction.
\end{proof}
The following is immediate from \cite[Proposition 5.1]{IRT12}.
\begin{corollary}
If $\Gamma$ has no cycles, i.e. is a tree, then for any initial condition $u_0\in L^2(X_\Gamma,\mu_\Gamma)$ the problem (\ref{E:vectorCauchy}) has a unique solution. 
\end{corollary}

We provide some rudimentary estimates.
\begin{corollary}\label{C:rudiestmetric}
Let $u_0$, $h_0$ and $u$ be as in Theorem \ref{T:exmetric}. Assume in addition that $h_0(s_0)=0$ for some $s_0\in X_{\Gamma}$. 
\begin{enumerate}
\item[(i)] We have $\sup_{t>0}\left\|u(t)\right\|_{L^2(X_\Gamma,\mu_\Gamma)}\leq c_1\left\|u_0\right\|_{L^2(X_\Gamma,\mu_\Gamma)}e^{c_2\left\|u_0\right\|_{L^2(X_\Gamma,\mu_\Gamma)}}$ with positive constants $c_1$ and $c_2$ independent of $u_0$.
\item[(ii)] If $\widetilde{u}_0=d\widetilde{h}_0$ is another initial condition with $\widetilde{h}_0\in W^{1,2}(X_\Gamma,\mu_\Gamma)$ such that $\widetilde{h}_0(s_0)=0$, and $\widetilde{u}$ the corresponding solution, then 
\begin{multline}
\sup_{t>0} \left\|u(t)-\widetilde{u}(t)\right\|_{L^2(X_\Gamma,\mu_\Gamma)}\notag\\
\leq c_3(\left\|u_0\right\|_{L^2(X_\Gamma,\mu_\Gamma)}+1)^2e^{c_4(\left\|u_0\right\|_{L^2(X_\Gamma,\mu_\Gamma)}+\left\|\widetilde{u}_0\right\|_{L^2(X_\Gamma,\mu_\Gamma)})}\left\|u_0-\widetilde{u}_0\right\|_{L^2(X_\Gamma,\mu_\Gamma)}
\end{multline}
with positive constants $c_3$ and $c_4$ independent of $u_0$ and $\widetilde{u}_0$.
\end{enumerate}
\end{corollary}
The proof relies on standard arguments, it is briefly sketched in the appendix.

\section{Heat and Burgers equations on the Sierpinski gasket}\label{S:SG}

In this section we present two formulations of Burgers equation on the Sierpinski gasket, they correspond to (\ref{E:scalarBurgers}) and (\ref{E:vectorBurgers}) in a similar way as (\ref{E:scalarCauchy}) and (\ref{E:vectorCauchy}). We then analyze the formulation of (\ref{E:vectorBurgers}) in more detail.

\subsection{Preliminaries on the Sierpinski gasket}

To provide some preliminaries we follow \cite{Ki01} and \cite{Str06}. The Sierpinski gasket $K$ is defined as the unique self-similar set determined by the contractive similarities $F_i:\mathbb{R}^2\to\mathbb{R}^2$ given by $F_i(x)=\frac12(x-q_i)+q_i$, $i=0,1,2$, where the $q_i$ 
are the vertices of an equilateral triangle of side length $1$. Let $V_0=\left\lbrace q_0,q_1,q_2\right\rbrace$. Given a word $w=w_1w_2...w_m$ of length $|w|=m$ with $w_i\in \left\lbrace 0,1,2\right\rbrace$ we write $F_w=F_{w_1}\circ F_{w_2}\circ ...\circ F_{w_m}$ and use the notations $K_w:=F_w(K)$ and $V_m:=\cup_{|w|=m} F_wV_0$ and $V_\ast:=\cup_{m\geq 0} V_m$.

For each $m\geq 0$ we consider $V_m$ as the vertex set of a (discrete) graph $G_m=(V_m,E_m)$ with vertices $p,q\in V_m$ being the endpoints of an edge $e\in E_m$ connecting them if there is a word $w$ of length $|w|=m$ such that $p,q\in F_w V_0$. In this case we write $p\sim_m q$. One can define a nontrivial quadratic form $\mathcal{E}$ acting on functions $f:V_\ast\to\mathbb{R}$ as the limit 
\[\mathcal{E}(f):=\lim_{m\to\infty}\mathcal{E}_m(f)\]
of rescaled graph energies 
\[\mathcal{E}_m(f)=r^{-m}\sum_{p\in V_m}\sum_{q\sim_m p}(f(p)-f(q))^2\]
along this sequence $(G_m)_{m\geq 0}$ of graphs. The rescaling factor is given by $r:=\frac35$. On $V_\ast$ we can define an associated metric by $R(p,q):=\sup\left\lbrace |f(p)-f(q)|^2:\: f:V_\ast\to\mathbb{R}, \mathcal{E}(f)<+\infty\right\rbrace$ for $p,q\in V_\ast$, and then recover $K$ as the completion of $V_\ast$ in this metric. The space $K$ endowed with the unique continuation of this metric, again denoted by $R$, is compact. Each function $f$ on $V_\ast$ with $\mathcal{E}(f)<+\infty$ extends to a continuous function on $K$, again denoted by $f$. Defining $\mathcal{D}(\mathcal{E})$ to be the set of functions $f:K\to\mathbb{R}$ such that $\mathcal{E}(f):=\mathcal{E}(f|_{V_\ast})<+\infty$, we obtain a resistance form $(\mathcal{E},\mathcal{D}(\mathcal{E}))$ on $K$ in the sense of \cite[Definition 2.8]{Ki03}. The elements of $\mathcal{D}(\mathcal{E})$ are continuous on $K$, hence bounded, and $\mathcal{D}(\mathcal{E})$ is an algebra with pointwise multiplication. Moreover, for any fixed $m$ we have
\[\mathcal{E}(f)=\sum_{|w|=m} \mathcal{E}_{K_w}(f),\quad f\in\mathcal{D}(\mathcal{E}),\]
where $\mathcal{E}_{K_w}$ denotes the restriction of $\mathcal{E}$ to $K_w$, see for instance \cite[Proposition 3.6]{PS17}. In the present context we have $\mathcal{E}_{K_w}(f)=r^{-m}\mathcal{E}(f\circ F_w)$, $f\in\mathcal{D}(\mathcal{E})$. For a fixed word $w$ of length $|w|=m$ the form $\mathcal{E}_{K_w}$ satisfies
\begin{equation}\label{E:resistKw}
(f(x)-f(y))^2\leq r^m\:\mathcal{E}_{K_w}(f)
\end{equation}
for any $f\in\mathcal{D}(\mathcal{E})$ and any $x,y\in K_w$.
 
Basically following \cite{CS03, CS09, IRT12} we can introduce a first order derivation $\partial$ associated with $(\mathcal{E},\mathcal{D}(\mathcal{E}))$.  Let $l_a(K\times K)$ denote the space of all real valued antisymmetric functions on $K\times K$. Given $v\in l_a(K\times K)$ and $g\in\mathcal{D}(\mathcal{E})$ we can define a new element $gv$ of $l_a(K\times K)$ by 
\begin{equation}\label{E:action}
(gv)(x,y):=\overline{g}(x,y)v(x,y),\quad x,y\in K, 
\end{equation}
where $\overline{g}(x,y):=\frac12(g(x)+g(y))$. This defines an action of $\mathcal{D}(\mathcal{E})$ on $l_a(K\times K)$ making it into a $\mathcal{D}(\mathcal{E})$-module. Next, let $d_u:\mathcal{D}(\mathcal{E})\to l_a(K\times K)$ be the universal derivation defined by 
\begin{equation}\label{E:universalder}
d_uf(x,y):=f(x)-f(y), \quad x,y\in K.
\end{equation}
It satisfies $\left\|d_uf\right\|_{\mathcal{H}}^2=\mathcal{E}(f)$ for any $f\in\mathcal{D}(\mathcal{E})$, and by (\ref{E:action}) also $d_u(fg)=fd_ug+gd_uf$ for any $f,g\in\mathcal{D}(\mathcal{E})$. 
Now let $\Omega_a^1(K)$ denote the submodule of $l_q(K\times K)$ generated by the functions of form $gd_uf$. On $\Omega_a^1(K)$ we can introduce a symmetric nonnegative definite bilinear form $\left\langle\cdot,\cdot\right\rangle_{\mathcal{H}}$ by extending 
\[\left\langle g_1d_uf_1, g_2d_uf_2\right\rangle_{\mathcal{H}}:=\lim_{m\to\infty} r^{-m}\sum_{p\in V_m}\sum_{q\sim_m p} \overline{g_1}(p,q)\overline{g_2}(p,q)d_uf_1(p,q)d_uf_2(p,q)\]
linearly in both arguments, respectively. Factoring out zero seminorm elements and completing yields a Hilbert space $\mathcal{H}$ to which we refer as the \emph{space of generalized $L^2$-vector fields associated with $(\mathcal{E},\mathcal{D}(\mathcal{E}))$}. 

\begin{remark}
The elements $v$ of $\mathcal{H}$ can no longer be interpreted as functions, in fact, for classical setups such as Euclidean spaces or Riemannian manifolds the space $\mathcal{H}$, defined in different but equivalent way, is the space of square integrable vector fields, see \cite{CS03, HRT13} or \cite{HT-fgs5}.
\end{remark}

The action (\ref{E:action}) induces an action of $\mathcal{D}(\mathcal{E})$ on $\mathcal{H}$ which satisfies 
\begin{equation}\label{E:boundedaction}
\left\|gv\right\|_{\mathcal{H}}\leq \left\|g\right\|_{\sup}\left\|v\right\|_{\mathcal{H}}
\end{equation}
for all $g\in\mathcal{D}(\mathcal{E})$ and all $v\in\mathcal{H}$. Given $f\in\mathcal{D}(\mathcal{E})$, we denote the $\mathcal{H}$-equivalence class of the universal derivation $d_uf$ as in (\ref{E:universalder}) by $\partial f$. This defines a derivation operator 
\[\partial:\mathcal{D}(\mathcal{E})\to\mathcal{H}\] 
that satisfies $\left\|\partial f\right\|_{\mathcal{H}}^2=\mathcal{E}(f)$ for any $f\in\mathcal{D}(\mathcal{E})$ and $\partial(fg)=f\partial g+g\partial f$ for any $f,g\in \mathcal{D}(\mathcal{E})$.

\begin{remark}\label{R:isoisoSG}
Again we have $\ker \partial =\mathbb{R}$, and the spaces $\im \partial$ and $\mathcal{D}(\mathcal{E})/\mathbb{R}$ are isometrically isomorphic.
\end{remark}

\begin{remark}
For Euclidean domains or Riemannian manifolds the operator $\partial$, defined in a different but equivalent way, yields the usual gradient operator, see  \cite{CS03, HRT13, HT-fgs5}. For a compact metric graph $\Gamma$ as in the preceding section one can similarly define a Hilbert space $\mathcal{H}_\Gamma$, then based on the resistance form $(\mathcal{E}_{\Gamma},\dot{W}^{1,2}(X_\Gamma))$. In this case there is an isometric isomorphism from $\mathcal{H}_\Gamma$ onto $L^2(X_\Gamma,\mu_\Gamma)$ which takes $\partial f$ into $f'$, \cite[Proposition 5.1]{BK16}.
\end{remark}

In what follows let $\mu$ be an atom free nonnegative Radon measure on $K$ with full support. 

\subsection{Scalar Laplacian}
The form $(\mathcal{E},\mathcal{D}(\mathcal{E}))$ is a local regular Dirichlet form on $L^2(K,\mu)$, see for instance \cite[Theorem 3.4.6]{Ki01} or \cite[Theorem 9.4]{Ki12}. The space $\mathcal{D}(\mathcal{E})$, endowed with
\begin{equation}\label{E:DformdomainSG}
\left\|f\right\|_{\mathcal{D}(\mathcal{E})}:=(\mathcal{E}(f)+\left\|f\right\|_{L^2(K,\mu)})^{1/2},
\end{equation}
is a Hilbert space, we write $\left\langle \cdot,\cdot\right\rangle_{\mathcal{D}(\mathcal{E})}$ for its scalar product.
The inequality
\begin{equation}\label{E:supboundSG}
\left\|f\right\|_{\sup}\leq c\:\left\|f\right\|_{\mathcal{D}(\mathcal{E})},\quad f\in\mathcal{D}(\mathcal{E}),
\end{equation}
holds with a universal constant $c>0$, \cite[Lemma 5.2.8]{Ki01}. The generator $(\mathcal{L},\mathcal{D}(\mathcal{L}))$ of this form is denoted by $(\mathcal{L},\mathcal{D}(\mathcal{L}))$. The derivation $\partial$ extends to a closed unbounded operator $\partial:L^2(K,\mu)\to \mathcal{H}$ with domain $\mathcal{D}(\mathcal{E})$. Its adjoint is denoted by $\partial^\ast$ and its domain by $\mathcal{D}(\partial^\ast)$. The image $\im \partial$ of the derivation $\partial$ is a closed subspace of $\mathcal{H}$ and we observe the orthogonal Helmholtz-Hodge type decomposition 
\begin{equation}\label{E:Helmholtz}
\mathcal{H}=\im \partial \oplus \ker \partial^\ast.
\end{equation}

\subsection{Vector Laplacian}
Similarly as in the metric graph case we can introduce a closed quadratic form $(\vec{\mathcal{E}},\mathcal{D}(\vec{\mathcal{E}}))$ on the Hilbert space $\mathcal{H}$ by setting $\mathcal{D}(\vec{\mathcal{E}}):=\mathcal{D}(\partial^\ast)$ and 
\[\vec{\mathcal{E}}(u,v):=\left\langle \partial^\ast u, \partial^\ast v\right\rangle_{L^2(K,\mu)},\quad u,v\in \mathcal{D}(\vec{\mathcal{E}}).\]
The associated generator is $(\vec{\mathcal{L}},\mathcal{D}(\vec{\mathcal{L}}))$, and $v\in \mathcal{H}$ is in $\mathcal{D}(\vec{\mathcal{L}})$ if and only if $\partial^\ast v\in \mathcal{D}(\mathcal{E})$. In this case, we have $\vec{\mathcal{L}}v=-\partial \partial^\ast v$. 

\subsection{Distributional definitions}

Similarly as before we consider $\partial^\ast$ and $\mathcal{L}$ also in the distributional sense as bounded linear operators $\partial^\ast: \mathcal{H}\to (\mathcal{D}(\mathcal{E}))^\ast$ and $\mathcal{L}:\mathcal{D}(\mathcal{E})\to(\mathcal{D}(\mathcal{E}))^\ast$ by $\partial^\ast v(\varphi):=\left\langle v,\partial\varphi\right\rangle_{\mathcal{H}}$ and $\mathcal{L}f(\varphi):=-\mathcal{E}(f,\varphi)$. Using the norm $v\mapsto \left\|\partial^\ast v\right\|_{\mathcal{D}(\mathcal{E})}$ on $\mathcal{D}(\vec{\mathcal{L}})$ we can see that the operator $\vec{\mathcal{L}}$ induces a bounded linear operator $\vec{\mathcal{L}}: L^2(K,\mu)\to (\mathcal{D}(\vec{\mathcal{L}}))^\ast$, defined by 
\[\vec{\mathcal{L}}v(w):=\partial^\ast v(\partial^\ast w),\quad w\in \mathcal{D}(\vec{\mathcal{L}}).\] 
We can generalize the former definition of the convection term $d(u^2)$ in (\ref{E:vectorCauchy}) by defining $\partial\left\langle u,u\right\rangle\in (\mathcal{D}(\vec{\mathcal{L}}))^\ast$ for any $u\in\mathcal{H}$ via
\begin{equation}\label{E:convect}
\partial \left\langle u,u\right\rangle(v):=\left\langle (\partial^\ast v)u,u\right\rangle_{\mathcal{H}},\quad v\in\mathcal{D}(\vec{\mathcal{L}}).
\end{equation}

\subsection{Hodge star operators and scalar Burgers equation}

The formulation of a counterpart of (\ref{E:scalarBurgers}) and (\ref{E:scalarCauchy}) on the Sierpinski gasket is non-trivial, note that a priori $\partial(g^2)$ is not a scalar function. However, since the Sierpinski gasket is one-dimensional in a certain way, \cite{Ku89, Hino10}, the gradient field $\partial(g^2)$ can be interpreted as function. We take a short detour to make this precise. Given a function $f\in\mathcal{D}(\mathcal{E})$ we can define its \emph{energy measure $\nu_f$} by the requirement that 
\[\int_K g\:d\nu_f=\mathcal{E}(fg,f)-\frac12\mathcal{E}(f^2,g), \quad g\in\mathcal{D}(\mathcal{E}),\]
see \cite[Section 3.2]{FOT94}. For any word $w$ of length $m$ it satisfies $\nu_f(K_w)=\mathcal{E}_{K_w}(f)$, what provides an alternative, equivalent definition, \cite{Str06}. Energy measures connect to the space $\mathcal{H}$ by the identity $\left\|g\partial f\right\|_{\mathcal{H}}^2=\int_K g^2\:d\nu_f$ for any $f,g\in\mathcal{D}(\mathcal{E})$.

The space of functions on $K$ harmonic on $K\setminus V_0$ is a Hilbert space with inner product $\mathcal{E}$, its dimension is two, \cite{Str06}. Let $\left\lbrace h_1,h_2\right\rbrace$ be an orthonormal basis in this space. The measure $\nu:=\nu_{h_1}+\nu_{h_2}$ does not depend on the choice of this orthonormal basis, and it is called the \emph{Kusuoka measure}. The energy measures $\nu_f$ of any function $f\in\mathcal{D}(\mathcal{E})$ is absolutely continuous with respect to $\nu$ and therefore has a density $\Gamma(f)\in L^1(K,\nu)$. As shown in \cite[Section 2]{HRT13} there exists a measurable field $(\mathcal{H}_x)_{x\in K}$ of Hilbert spaces $(\mathcal{H}_x, \left\|\cdot,\cdot\right\|_{\mathcal{H}_x})$ such that 
the space $\mathcal{H}$ is isometrically isomorphic to the direct integral $\int_K^\oplus \mathcal{H}_x\:\nu(dx)$ with respect to $\nu$. The Hilbert spaces $\mathcal{H}_x$ may be seen as abstract substitutes of tangent spaces, and in particular, we have $\left\|\partial f\right\|_{\mathcal{H}_x}^2=\Gamma(f)(x)$ for $\nu$-a.e. $x\in K$. A very well known observation, basically due to Kusuoka, \cite{Ku89}, and studied further in \cite{Hino03, Hino05, Hino10}, is that for $\nu$-a.e. $x\in K$ the space $\mathcal{H}_x$ is one-dimensional. 

\begin{remark}
To have this direct integral representation the given volume measure $\mu$ does not have to be the Kusuoka measure $\nu$. In fact, a standard choice of volume measure is to take the natural (normalized) self-similar Hausdorff measure of dimension $\frac{\log 3}{\log 2}$, and the Kusuoka measure is singular with respect to it, \cite{BBST, Hino03, Hino05}.
\end{remark}

It is not difficult to find an element $\omega$ of $\mathcal{H}$ that satisfies $\left\|\omega\right\|_{\mathcal{H}_x}=1$ for $\nu$-a.e. $x\in K$, \cite[Lemma 4.3]{BK16}. Consequently for any $v\in\mathcal{H}$ there is a uniquely defined function $g\in L^2(K,\nu)$ such that $v=g\omega$. It was shown in \cite[Section 4]{BK16} that the map $\star_\omega: \mathcal{H}\to L^2(K,\mu)$, defined by 
\[\star_\omega v:=g,\]
provides an isometric isomorphism from $\mathcal{H}$ onto $L^2(K,\nu)$, \cite[Proposition 4.5]{BK16}. To $\star_\omega$ we refer as the \emph{Hodge star operator associated with $\omega$}, \cite[Definition 4.4]{BK16}. In the following, let $\omega\in \mathcal{H}$ with $\left\|\omega\right\|_{\mathcal{H}_x}=1$ $\nu$-a.e. be fixed.

Mathematically it now seems reasonable to formulate (\ref{E:scalarBurgers}) as the Cauchy problem
\begin{equation}\label{E:scalarCauchySG}
\begin{cases} g_t(t)&=\mathcal{L}g(t)-\frac12\star_\omega \partial(g^2)(t),\\
g(0)&=g_0.\end{cases}
\end{equation}
In \cite{LiuQian17} the authors make heavy use of (\ref{E:resistKw}) respectively (\ref{E:supboundSG}) and skillfully establish Sobolev inequalities on $K$ for mutually singular measures. They combine them with known results on heat kernels to obtain the existence and uniqueness of weak solutions, \cite[Definition 4.13]{LiuQian17}, to a counterpart of (\ref{E:scalarCauchySG}) subject to Dirichlet boundary conditions and in the case that $\mu$ is the natural Hausdorff measure on $K$, \cite[Theorem 4.16]{LiuQian17}. Their results work for arbitrary finite time intervals $[0,T]$. Without mentioning it explicitely, they make use of a Hodge star operator $\star_\omega$. In fact, in a probabilistic form it already appeared in \cite[Theorem 5.4 (ii)]{Ku89} (for $p=1$), as can be seen using Nakao's theorem, see for example \cite[Theorem 9.1]{HRT13}. Under additional conditions also well known semigroup methods may be applied to obtain existence and uniqueness of solutions to (\ref{E:scalarCauchySG}), \cite[Section 6.3, Theorem 3.1]{P83}, at least for small time intervals.

\subsection{Vector Burgers equation}
Here we focus on (\ref{E:vectorBurgers}). Suppose that $w(t)=e^{t\mathcal{L}}w_0$ is the unique solution to the heat equation (\ref{E:Kirchhoffheat}), now for the generator $(\mathcal{L},\mathcal{D}(\mathcal{L}))$ of $(\mathcal{E},\mathcal{D}(\mathcal{E}))$. Again we assume the initial condition $w_0\in L^2(K,\mu)$ to be strictly nonnegative $\mu$-a.e. and strictly positive on a set of positive measure. The $\mathcal{D}(\mathcal{E})$-valued function $h:=-2\log w$ satisfies (\ref{E:metricKPZ}) for any $\varphi\in \mathcal{D}(\mathcal{E})$. We write $\partial\left\langle u,u\right\rangle(t):=\partial\left\langle u(t),u(t)\right\rangle$, the latter defined as in (\ref{E:convect}), and consider the Cauchy problem
\begin{equation}\label{E:vectorCauchySG}
\begin{cases}
u_t(t) & =\vec{\mathcal{L}}u(t)-\frac12\partial\left\langle u,u \right\rangle(t),\quad t>0,\\
u(0) & =u_0.
\end{cases}
\end{equation} 
The definition of solution is similar to the metric graph case.
\begin{definition}
A function $u\in C([0,+\infty), \mathcal{H})\cap C^1((0,+\infty),\mathcal{H})$ is called a solution to (\ref{E:vectorCauchySG}) with initial condition $u_0\in\mathcal{H}$ if $u$ satisfies the first identity in (\ref{E:vectorCauchySG}) in $(\mathcal{D}(\vec{\mathcal{L}}))^\ast$ and the second in $\mathcal{H}$.
\end{definition}

The structure of solutions is as in the metric graph case. 

\begin{theorem}\label{T:SGstructure}
Suppose $u$ is a solution to (\ref{E:vectorCauchySG}) with initial condition $u_0\in\mathcal{H}$. Let $\eta_0\in \ker \partial^\ast$ and $h_0\in \mathcal{D}(\mathcal{E})$ be such that $u_0=\partial h_0+\eta_0$. Then $u$ is of form $u(t)=\partial h(t)+\eta_0$, $t\geq 0$, with a function $h:[0,+\infty)\to \mathcal{D}(\mathcal{E})$ satisfying $dh(0)=dh_0$.
\end{theorem}
\begin{proof}
Considering $\partial$ and $\partial^\ast$ in place of $d$ and $d^\ast$, respectively and using (\ref{E:Helmholtz}) we can follow the proof of Theorem \ref{T:metricstructure}.
\end{proof}

Again we can conclude an existence and uniqueness statement for solutions.

\begin{theorem}\label{T:exSG}
Assume $u_0=\partial h_0$ with $h_0\in\mathcal{D}(\mathcal{E})$. Let $w(t)$ denote the unique solution $e^{t\mathcal{L}}w_0$ to (\ref{E:Kirchhoffheat}) with initial condition $w_0:=e^{-h_0/2}$. Then the function $u(t):=-2\partial\log w(t)$, $t\geq 0$, is the
unique solution to (\ref{E:vectorCauchySG}). 
\end{theorem}
\begin{proof}
The proof is similar to that of Theorem \ref{T:exmetric}, note that the continuity of the Cole-Hopf solution at zero can be verified using (\ref{E:boundedaction}) and the uniqueness follows from (\ref{E:Helmholtz}) together with Theorem \ref{T:SGstructure}.
\end{proof}

Also the following estimates are as before, see the appendix for a proof.
\begin{corollary}\label{C:rudiestSG}
Let $u_0$, $h_0$ and $u$ be as in Theorem \ref{T:exSG}. Assume in addition that $h_0(x_0)=0$ for some $x_0\in K$. 
\begin{enumerate}
\item[(i)] We have $\sup_{t>0}\left\|u(t)\right\|_{\mathcal{H}}\leq c_1\left\|u_0\right\|_{\mathcal{H}}e^{c_2\left\|u_0\right\|_{\mathcal{H}}}$ with positive constants $c_1$ and $c_2$ independent of $u_0$.
\item[(ii)] If $\widetilde{u}_0=d\widetilde{h}_0$ is another initial condition with $\widetilde{h}_0\in \mathcal{D}(\mathcal{E})$ such that $\widetilde{h}_0(x_0)=0$, and $\widetilde{u}$ the corresponding solution, then 
\[\sup_{t>0} \left\|u(t)-\widetilde{u}(t)\right\|_{\mathcal{H}}
\leq c_3(\left\|u_0\right\|_{\mathcal{H}}+1)^2e^{c_4(\left\|u_0\right\|_{\mathcal{H}}+\left\|\widetilde{u}_0\right\|_{\mathcal{H}})}\left\|u_0-\widetilde{u}_0\right\|_{\mathcal{H}}\]
with positive constants $c_3$ and $c_4$ independent of $u_0$ and $\widetilde{u}_0$.
\end{enumerate}
\end{corollary}

%

\begin{remark}
It is well known that in the context of classical partial differential equations the Cole-Hopf transform connects an entire hierarchy of equations and allows to obtain exact solutions to non-linear equations from solutions to linear equations on each particular level, \cite{KuSh09, Tasso76}. On fractals linear second order ('heat') equations (\ref{E:Kirchhoffheat}) are tractable whenever we can understand a natural Laplace operator. In comparison, linear first order ('transport') equations of type $g_t=g_x$ are more difficult to analyze, and due to possible energy singularity the existing methods, such as \cite{AT14},  may work for some volume measures, but certainly not for all. Linear equations of higher order, for instance $g_t=g_{xxx}$, have not yet been studied on fractals, and it is an interesting open question how to formulate them in a meaningful way.
\end{remark}

\begin{remark}
Formulation (\ref{E:vectorCauchySG}) and Theorem \ref{T:exSG} only rely on basic ingredients such as the Markov property and the conservativity of the semigroup. In a similar manner one could study Burgers type equations via the Cole-Hopf transform for a large class of Dirichlet forms, for instance for purely non-local forms. However, the physical relevance of such models may of course remain a matter of discussion.
\end{remark}

\section{Metric graph approximation for Cole-Hopf solutions}\label{S:approx}

In this section we consider metric graph approximations to the Sierpinski gasket $K$. We show that a Cole-Hopf solution 
of the Burgers equation (\ref{E:vectorCauchySG}) on $K$ can be approximated in a suitable sense by Cole-Hopf solutions to Burgers equations (\ref{E:vectorCauchy}) on the metric graphs approximating $K$. Our approximation scheme follows the methods in \cite{PS17} and \cite{PS17a}.

Let $(\Gamma_m)_{m\geq 0}$ be the sequence of metric graphs $\Gamma_m=(E_m,V_m,i_m,j_m)$ naturally defined by the graphs $G_m=(V_m,E_m)$ approximating $K$ as in Section \ref{S:SG}. For simplicity we write $i=i_m$ and $j=j_m$. Note that for all $e\in E_m$ we have $l_e=2^{-m}$. On the space $X_m:=X_{\Gamma_m}$, defined as in (\ref{E:spaceX}), we consider the bilinear form $(\mathcal{E}_{\Gamma_m}, \dot{W}^{1,2}(X_m))$, where
\[\mathcal{E}_{\Gamma_m}(f):=2^{-m}r^{-m}\sum_{e\in E_m} \mathcal{E}_e(f_e) \quad \text{and}\quad \mathcal{E}_e(f_e)=\int_0^{2^{-m}} (f_e'(t))^2 dt.\]

To a function $f\in \dot{W}^{1,2}(X_m)$ which is linear on each edge $e\in E_m$ we refer as \emph{edge-wise linear} function, and we denote the subspace of $\dot{W}^{1,2}(X_m)$ of
such functions by $EL_m$. If $f\in EL_m$, then its derivative on $e$ is the constant function
$f_e'=2^m(f(j(e))-f(i(e)))$, so that 
\begin{equation}\label{E:harmoniccase}
\mathcal{E}_e(f_e)=\int_0^{2^{-m}}(f'_e(t))^2dt=2^m(f(j(e))-f(i(e)))^2
\end{equation}
on each $e\in E_m$ and consequently $\mathcal{E}_{\Gamma_m}(f)=\mathcal{E}_m(f|_{V_m})$. For a general function $f\in \dot{W}^{1,2}(X_m)$ formula (\ref{E:harmoniccase}) becomes an inequality in which the left hand side dominates the right hand side. This implies  
\begin{equation}\label{E:metricdominates}
\mathcal{E}_m(f|_{V_m})\leq \mathcal{E}_{\Gamma_m}(f), \quad f\in \dot{W}^{1,2}(X_m).
\end{equation}
By $H_{\Gamma_m}$ we denote the linear operator $H_{\Gamma_m}:\dot{W}^{1,2}(X_m)\to EL_m$ that assigns to a function $f\in \dot{W}^{1,2}(X)$ the unique edge-wise linear function on $X_m$ that interpolates $f|_{V_m}$.

A function $f$ on $K$ is called \emph{$m$-piecewise harmonic} if it minimizes all energies $\mathcal{E}_n$, $n\geq m+1$, amongst all functions on $K$ which coincide with $f|_{V_m}$ on $V_m$. If $f$ is $m$-piecewise harmonic, then it is also $n$-piecewise harmonic for any $n\geq m$, $f\in\mathcal{D}(\mathcal{E})$ and $\mathcal{E}(f)=\mathcal{E}_m(f|_{V_m})$.
For details see for instance \cite{Str06}. We write $PH_m$ for the subspace of all $m$-piecewise harmonic functions. As usual we denote by $\psi_{p,m}$ the function in $PH_m$ satisfying $\psi_{p,m}(q)=\delta_{pq}$, $q\in V_m$. Given a function $f$ on $V_m$ we write $H_m(f)\in\mathcal{D}(\mathcal{E})$ to denote its unique extension to an $m$-piecewise harmonic function, 
\begin{equation}\label{E:harmextop}
H_mf(x):=\sum_{p\in V_m} f(p)\psi_{p,m}(x),\quad x\in K.
\end{equation}
We use the same symbol $H_m$ to denote the linear operator $H_m:\mathcal{D}(\mathcal{E})\to PH_m$ defined by $H_m(f):=H_m(f|_{V_m})$, $f\in\mathcal{D}(\mathcal{E})$.

Given a function $f\in \mathcal{D}(\mathcal{E})$ on $K$, we can interpret its pointwise restriction to the line segment connecting two neighbor points $p\sim_m q$ from $V_m$ as a continuous function $f_e$ on the edge $e\in E_m$ of $\Gamma_m$ with $i(e)=p$ and $j(e)=q$. This defines a continuous function on $X_m$, which we denote by $f|_{X_m}$. Since a function $f\in PH_m$ is linear on all line segments connecting two neighbor points $p\sim_m q$, the above interpretation $f|_{X_m}$ of $f$
is a function in $EL_m$ which satisfies (\ref{E:harmoniccase}) on each edge and $\mathcal{E}_{\Gamma_m}(f|_{X_m})=\mathcal{E}_m(f|_{V_m})=\mathcal{E}(f)$. Moreover, we have $H_{\Gamma_m}(f|_{X_m})=H_m(f)|_{X_m}$ for any $f\in\mathcal{D}(\mathcal{E})$. Since
\begin{equation}\label{E:phsplines}
\lim_{m\to \infty} \mathcal{E}(H_m(f)-f)=0
\end{equation}
for any $f\in\mathcal{D}(\mathcal{E})$, see for instance \cite[Theorem 1.4.4]{Str06}, we observe that
\[\mathcal{E}(f)=\sup_m\mathcal{E}_{\Gamma_m}(H_m(f)|_{X_m}), \quad f\in\mathcal{D}(\mathcal{E}).\]

Now let $c$ be the function on $V_\ast$ defined by
\[c(p):=\begin{cases}\frac14 \quad \text{if $p\in V_\ast\setminus V_0$} \\ \frac12 \quad \text{if $p\in V_0$.}\end{cases}\]
Given an edge $e\in E_m$ we set
\[\psi_{e,m}(x):=c(i(e))\psi_{i(e),m}(x)+c(j(e))\psi_{j(e),m}(x),\quad x\in K,\]
and obtain a function $\psi_{e,m}$ which satisfies 
\begin{equation}\label{E:sumpsiem}
\sum_{e\in E_m}\left\langle \psi_{e,m},\psi_{e',m}\right\rangle_{L^2(K,\mu)}=\sum_{p\in V_0}\psi_{p,m}(x)+\sum_{p\in V_m\setminus V_0} \psi_{p,m}(x)= 1,\quad x\in K.
\end{equation}
We endow the space $X_m$ with the measure $\mu_m:=\mu_{\Gamma_m}$ defined as in (\ref{E:measure}) with constants 
\[c_e:= 2^m\left(\int_K \psi_{e,m}(x)\mu(dx)\right),\quad e\in E_m,\] 
so that $\mu_m(dt):=\sum_{e\in E_m} c_e\lambda_e(dt)$.
The average of a function $f\in L^2(X_m,\mu_m)$ on an edge $e\in E_m$ we denote by
\[\overline{f_e}:=2^m\int_0^{2^{-m}} f_e(s)\:ds.\]
We write $W^{1,2}(X_m,\mu_m)$ for the space $\dot{W}^{1,2}(X_m,\mu_m)$ endowed with the Hilbert norm as in (\ref{E:W12norm}) and consider the strongly local regular Dirichlet form $(\mathcal{E}_{\Gamma_m}, W^{1,2}(X_m,\mu_m))$
on $L^2(X_m,\mu_m)$.

%

\subsection{Approximation of solutions to the heat equations}
To study convergence statements we follow \cite{P12, PS17}. We define identification operators  $J_{0,m}:L^2(X_m,\mu_m)\to L^2(K,\mu)$ by 
\[J_{0,m} f(x):=\sum_{e\in E_m} \overline{f_e} \:\psi_{e,m}(x),\quad x\in K.\]

\begin{proposition}\label{P:bounded}
The operators $J_{0,m}$ satisfy $\left\|J_{0,m} f\right\|_{L^2(K,\mu)}\leq \left\|f\right\|_{L^2(X_m,\mu_m)}$ for any $f\in L^2(X_m,\mu_m)$. The adjoint $J_{0,m}^\ast:L^2(K,\mu)\to L^2(X_m,\mu_m)$ of $J_{0,m}$ is given by
\[J_{0,m}^\ast u(t)=\sum_{e\in E_m}\:\mathbf{1}_e(t)\:\frac{\left\langle u,\psi_{e,m}\right\rangle_{L^2(K,\mu)}}{\left(\int_K \psi_{e,m} d\mu\right)},\quad u\in L^2(K,\mu).\]
\end{proposition}

\begin{proof}
By Cauchy-Schwarz and (\ref{E:sumpsiem})
\begin{align}
\left\|J_{0,m} f\right\|_{L^2(X,\mu)}^2 &=\int_K\sum_{e,e'\in E_m}2^{2m}\int_0^{2^{-m}}\int_0^{2^{-m}}f_e(s)f_{e'}(s')\psi_{e,m}(x)\psi_{e',m}(x)dsds'\mu(dx)\notag\\
&\leq \frac12\sum_{e\in E_m} 2^{m}\int_0^{2^{-m}}f_e(s)^2ds \left(\int_K\psi_{e,m}(x)\mu(dx)\right) \notag\\
&\ \ \ \ + \frac12\sum_{e'\in E_m} 2^{m}\int_0^{2^{-m}}f_{e'}(s')^2ds' \left(\int_K\psi_{e',m}(x)\mu(dx)\right) \notag\\
&=\left\|f\right\|_{L^2(X_m,\mu_m)}^2.\notag
\end{align}
The second statement follows because for any $f\in L^2(X_m,\mu_m)$ and $u\in L^2(K,\mu)$ we have
\[\left\langle J_{0,m} f,u\right\rangle_{L^2(K,\mu)}=\sum_{e\in E_m}2^m\int_0^{2^{-m}}f_e(s)ds\left\langle \psi_{e,m},u\right\rangle_{L^2(K,\mu)}.\]
\end{proof}

The next convergence statement is a special case of \cite[Theorem 1.1]{PS17}.
\begin{theorem}\label{T:conv}
For  any $t>0$ we have 
\[\lim_m \left\|e^{t\mathcal{L}} - J_{0,m} e^{t\mathcal{L}_{\Gamma_m}} J_{0,m}^\ast \right\|_{L^2(K,\mu)\to L^2(K,\mu)}=0.\]
\end{theorem}

Theorem \ref{T:conv} will follow from the spectral convergence results in \cite{P12, PS17, PS17a}. That the necessary hypotheses are valid we verify in the appendix.

Theorem \ref{T:conv} can be used to see that in some way the solutions to the heat equations on the approximating spaces $X_m$ converge to the solution to the heat equation on $K$. We first collect some prerequisites.

\begin{lemma}\label{L:pre}
Given $w_0\in L^2(K,\mu)$. For any $m\geq 1$ let $w_m(t)$ denote the unique solution to (\ref{E:Kirchhoffheat}) for $\mathcal{L}_{\Gamma_m}$ in $L^2(X_m,\mu_m)$ with initial condition $J_{0,m}^\ast w_0$. Then we have $\sup_m \mathcal{E}_{\Gamma_m}(w_m(t))<+\infty$ for any $t>0$.
\end{lemma}

\begin{proof}
There is a constant $c>0$ independent of $m$ and $t$ such that for any $t>0$ we have
\begin{equation}\label{E:analyticbound}
\big\|\sqrt{\mathcal{L}_{\Gamma_m}}e^{t\mathcal{L}_{\Gamma_m}}\big\|_{L^2(X_m,\mu_m)\to L^2(X_m,\mu_m)}\leq c\: t^{-1/2},
\end{equation}
as follows from the spectral theorem: Since the metric graphs $\Gamma_m$ are compact, the operators $\mathcal{L}_{\Gamma_m}$ have pure point spectrum, \cite[Proposition 2.2.14]{P12}. Consequently the eigenvalues of $-\mathcal{L}_{\Gamma_m}$, ordered with multiplicities taken into account, are $0=\lambda_0<\lambda_1\leq \lambda_2\leq ...$ with only accumulation point $+\infty$, and 
\[-\mathcal{L}_{\Gamma_m}f=\sum_{k=0}^\infty \lambda_k(m)\left\langle \varphi_k(m),f\right\rangle_{L^2(X_m,\mu_m)}\varphi_k(m),\quad f\in L^2(X_m,\mu_m),\]
where $\varphi_k(m)$ are the eigenfunction of $-\mathcal{L}_{\Gamma_m}$ for the eigenvalue $\lambda_k(m)$. This yields
\[\big\|\sqrt{-\mathcal{L}_{\Gamma_m}}e^{t\mathcal{L}_{\Gamma_m}}f\big\|_{L^2(X_m,\mu_m)}^2=t^{-1}\sum_{k=0}t\lambda_k(m)e^{-2t\lambda_k(m)}|\left\langle \varphi_k(m),f\right\rangle_{L^2(X_m,\mu_m)}|^2,\]
and since the function $s\mapsto s\:e^{-2s}$ is bounded on $[0,+\infty)$, this implies (\ref{E:analyticbound}). By (\ref{E:analyticbound}),
\begin{align}
\sup_m \mathcal{E}_{\Gamma_m}(w_m(t))&=\sup_m\big\|\sqrt{-\mathcal{L}_{\Gamma_m}}e^{t\mathcal{L}_{\Gamma_m}}J_{0,m}^\ast w_0\big\|^2_{L^2(X_m,\mu_m)}\notag\\
&\leq c\:t^{-1/2}\sup_m\big\|J_{0,m}^\ast w_0\big\|_{L^2(X_m,\mu_m)}^2\notag\\
&\leq c\:t^{-1/2}\left\|w_0\right\|_{L^2(K,\mu)}^2.\notag
\end{align}
\end{proof}

\begin{corollary}\label{C:heatsolconv}
Let $w_0$ and $w_m(t)$ be as in Lemma \ref{L:pre} and let $w(t)$ be the unique solution to (\ref{E:Kirchhoffheat}) for $\mathcal{L}$ in $L^2(K,\mu)$ with initial condition $w_0$.
\begin{enumerate}
\item[(i)] For any $t>0$ we have $\lim_{m\to \infty} H_m(w_m(t)|_{V_m})=w(t)$ strongly in $L^2(K,\mu)$ and weakly in $\mathcal{D}(\mathcal{E})$.
\item[(ii)] If $w_0\in\mathcal{D}(\mathcal{E})$ and $w_0$ is strictly positive on $K$ then for any $t>0$ we also have $\lim_{m\to \infty} H_m(\log w_m(t)|_{V_m})=\log w(t)$ strongly in $L^2(K,\mu)$ and weakly in $\mathcal{D}(\mathcal{E})$.
\end{enumerate}
\end{corollary}

\begin{remark} \mbox{}
\begin{enumerate}
\item[(i)] To obtain Corollary \ref{C:heatsolconv} it would be sufficient to verify convergence of the semigroup operators in the strong sense. However, practically it seems easier to verify generalized norm resolvent convergence in the sense of \cite{P12, PS17, PS17a} than to verify generalized Mosco convergence, \cite[Section 2.5]{KS03}, which would be equivalent to convergence of operators in a suitable strong sense, \cite[Theorem 2.4]{KS03}.
\item[(ii)] Considering $H_m(w_m(t)|_{V_m})$ we implicitely linearize $w_m(t)$ along the edges $E_m$ of $\Gamma_m$ and compare the resulting function to $w(t)$. Doing so, we discard information, but since we rely on approximation by functions from $PH_m$  anyway, (\ref{E:phsplines}), it is natural to proceed this way.
\item[(iii)] For the special case that $\mu$ is the natural self-similar Hausdorff measure on $K$ one can use higher order splines to approximate functions in, roughly speaking, the graph norm of the associated Laplacian, \cite[Theorem 7.5]{StrU00}. See also \cite{Str00} for related results. It will be a future project to try to combine this with a metric graph approximation scheme to obtain (strong) convergence in $\mathcal{D}(\mathcal{E})$ instead of in $L^2(K,\mu)$.
\end{enumerate}
\end{remark}

We prove Corollary \ref{C:heatsolconv}. 

\begin{proof} 
From Lemma \ref{L:compare01} (i) below in the appendix it follows that 
\[\big\|H_m(w_m(t)|_{V_m})-J_{0,m}w_m(t)\big\|_{L^2(K,\mu)}\leq 36\: r^m\:\max_{|w|=m}\mu(K_w)\:\mathcal{E}_{\Gamma_m}(w_m(t)),\]
and combining with  Lemma \ref{L:pre} we obtain
\[\lim_{m\to\infty}\big\|H_m(w_m(t)|_{V_m})-J_{0,m}w_m(t)\big\|_{L^2(K,\mu)}=0.\]
The $L^2(K,\mu)$ limit relation in Corollary \ref{C:heatsolconv} (i) now follows from Theorem \ref{T:conv}. By 
(\ref{E:metricdominates}) and Lemma \ref{L:pre} we also have $\sup_m \mathcal{E}(H_m(w_m(t)|_{V_m}))=\sup_m \mathcal{E}_m(w_m(t))<+\infty$.
Combining, we obtain $\sup_m \left\|H_m(w_m(t)|_{V_m})\right\|_{\mathcal{D}(\mathcal{E})}<+\infty$. Consequently, any fixed subsequence of $(H_m(w_m(t)|_{V_m}))_m$ has a further subsequence converging weakly in $\mathcal{D}(\mathcal{E})$,
we denote the limit by $\tilde{w}\in\mathcal{D}(\mathcal{E})$. By the Banach-Saks theorem, it has a subsequence whose convex combinations converge strongly to $\tilde{w}$ in $\mathcal{D}(\mathcal{E})$, hence also strongly in $L^2(K,\mu)$, which implies that $\tilde{w}$ must equal $w(t)$. This argument also shows that $(H_m(w_m(t)|_{V_m}))_m$ cannot have any other weak accumulation point than $w(t)$, what shows (i). To see (ii) suppose that there exists $\gamma>0$ such that $\inf_{x\in K}w_0(x)\geq \gamma$. As $(e^{t\mathcal{L}})_{t>0}$ is conservative and $w(t)\in\mathcal{D}(\mathcal{E})$ continuous we also have $\inf_{x\in K}w(t,x)\geq \gamma$ for any $t\geq 0$. The definition of the operators $J_{0,m}^\ast$, the conservativity of the semigroups $(e^{t\mathcal{L}_{\Gamma_m}})_{t>0}$ and the continuity of the functions $w_m(t)\in W^{1,2}(X_m,\mu_m)$ imply $\inf_{x\in K}w(t,x)\geq \gamma$ for any $m$ and any $t\geq 0$. These lower bounds imply
\begin{equation}\label{E:Hmbound}
\mathcal{E}(H_m(\log w_m(t)|_{V_m})=\mathcal{E}_m(\log w_m(t))\leq \gamma^{-2}\mathcal{E}_m(w_m(t)|_{V_m})\leq \sup_m\gamma^{-2}\mathcal{E}_{\Gamma_m}(w_m(t))
\end{equation}
Now let $\varepsilon>0$ be arbitrary and $m$ large enough so that 
\[r^m\leq \varepsilon\gamma\:\big\lbrace \sup_m\mathcal{E}_{\Gamma_m}(w(t))^{1/2}+\mathcal{E}(w(t))^{1/2}\big\rbrace^{-1}.\]
For any word $w$ with $|w|=m$ and any $x\in K_w$ estimate (\ref{E:resistKw}) then yields
\begin{align}
|H_m(\log & \:w_m(t)|_{V_m})(x)-\log w(t)(x)|\notag\\
&\leq |H_m(\log w_m(t)|_{V_m})(x)-H_m(\log w_m(t)|_{V_m})(p)| +|\log w(t)(x)-\log w(t)(p)|\notag\\
&\leq \diam(F_wK)\left\lbrace \mathcal{E}(H_m(\log w_m(t)|_{V_m}))^{1/2}+\mathcal{E}(w(t))^{1/2}\right\rbrace\notag\\
&\leq \varepsilon,\notag
\end{align}
where $p$ is a point from $V_m\cap K_w$. We have used (\ref{E:Hmbound}) and that $H_m(\log w_m(t)|_{V_m})(p)=\log w_m(t)(p)$ for all $p\in V_m$. As a consequence,
\[\left\|H_m(\log w_m(t)|_{V_m})-\log w(t)\right\|_{L^2(K,\mu)}^2=\sum_{|w|=m}\int_{K_w}|H_m(\log w_m(t)|_{V_m})-\log w(t)|^2d\mu\leq \varepsilon^2 \]
whenever $m$ is sufficiently large. Using (\ref{E:Hmbound}) we can proceed similarly as in (i) to see the weak convergence in $\mathcal{D}(\mathcal{E})$.
\end{proof}

\subsection{Approximation of Cole-Hopf solutions to the Burgers equation}

To formulate an approximation result for solutions to (\ref{E:vectorCauchy}) on $K$ by corresponding solutions to  (\ref{E:vectorCauchySG}) on the metric graphs $\Gamma_m$ we define the operators $H_{\Gamma_m}$ and $H_m$
on gradient fields. Since for any $c\in\mathbb{R}$ we have $H_{\Gamma_m}(f+c)=H_{\Gamma_m}(f)+c$, $f\in \dot{W}^{1,2}(X_m)$, and $H_m(f+c)=H_m(f)+c$, $f\in \mathcal{D}(\mathcal{E})$, these operators may be interpreted as linear operators on $\dot{W}(X_m)/\mathbb{R}$ and $\mathcal{D}(\mathcal{E})/\mathbb{R}$, respectively. According to Remarks \ref{R:isoiso} and \ref{R:isoisoSG} these spaces are isometrically isomorphic to $\im d$ and $\im \partial$, respectively, so that we obtain well-defined operators $H_{\Gamma_m}:\im d\to d(EL_m)$ and $H_m: \im \partial \to \partial (PH_m)$ by setting
\[H_{\Gamma_m}(df):=dH_{\Gamma_m}(f) \quad \text{ and } \quad H_m(\partial f):=\partial H_m(f).\] 
Moreover, for any $m$ we can define a vector space isomorphism $\Phi_m:EL_m\to PH_m$ by $\Phi_m(f):=H_m(f|_{V_m})$, its inverse $\Phi_m^{-1}$ is given by $f|_{X_m}$. It satisfies $\Phi_m(f+c)=\Phi_m(f)+c$, $c\in\mathbb{R}$, and therefore also induces a well defined linear map $\Phi_m: d(EL_m)\to \partial (PH_m)$ by
\[\Phi_m(df):=\partial \Phi_m(f), \quad f\in EL_m.\]
Since $\left\|\partial \Phi_m(f)\right\|_{\mathcal{H}}^2=\mathcal{E}(H_m(f|_{V_m}))=\mathcal{E}_{\Gamma_m}(f)=\left\|df\right\|_{L^2(X_m,\mu_m)}^2$ for any $f\in \dot{W}^{1,2}(X_m)$, the map $\Phi_m$ is seen to be an isometric isomorphism. 

For the solutions to the Burgers equations we now obtain the following result.
\begin{theorem}
Assume $u_0=\partial h_0$ with $h_0\in\mathcal{D}(\mathcal{E})$. Let $u(t)$ denote the unique solution to (\ref{E:vectorCauchySG}) with initial condition $u_0$ and for any $m\geq 1$ let $u_m(t)$ denote the unique solution to (\ref{E:vectorCauchy}) with initial condition $-2d\log J_{0,m}^\ast e^{-h_0/2}$. Then we have
\begin{equation}\label{E:weaklimitu}
\lim_{m\to\infty} \left\langle \Phi_m\circ H_{\Gamma_m}(u_m(t))-u(t),v\right\rangle_{\mathcal{H}}=0   
\end{equation}
for any $t\geq 0$ and $v\in\mathcal{H}$.
\end{theorem} 

\begin{proof}
For any $m\geq 1$ we have 
\begin{multline}
\Phi_m\circ H_{\Gamma_m}(u_m(t))=-2\Phi_m (d(H_{\Gamma_m}(\log w_m(t)))\notag\\
=-2\partial\Phi_m(H_{\Gamma_m}(\log(w_m(t)))=-2\partial H_m(\log(w_m(t)|_{V_m}),
\end{multline}
and according to Corollary \ref{C:heatsolconv} (ii),
\[\lim_{m\to\infty}\left\langle \partial H_m(\log w_m(t)|_{V_m})-\partial \log w(t),\partial \varphi\right\rangle_{\mathcal{H}}=\lim_{m\to\infty}\mathcal{E}(H_m(\log w_m(t)|_{V_m})-\log w(t),\varphi)=0\]
for any $\varphi\in\mathcal{D}(\mathcal{E})$. Since $\Phi_m\circ H_{\Gamma_m}(u_m(t))$ and $u(t)$ are elements of $\im \partial$, it follows from (\ref{E:Helmholtz}) that we may use general test vector fields $v\in\mathcal{H}$ in place of $\partial\varphi$.
\end{proof}

\begin{remark}
The space $\mathcal{H}$ can be rewritten as the closure of the union of an increasing sequence of finite dimensional subspaces, \cite[Definition 5.2, Lemmas 5.3 and 5.5 and Theorem 5.6]{IRT12}. Then (\ref{E:weaklimitu}) can also  be expressed using these subspaces.
\end{remark}

%
%


\section*{Appendix}

\subsection{Basic estimates}
We sketch a proof for Corollary \ref{C:rudiestmetric}, Corollary \ref{C:rudiestSG} follows similarly. Since $(\mathcal{E}_\Gamma,\dot{W}^{1,2}(X_\Gamma))$ is a resistance form (or, alternatively, using (\ref{E:resistonedge})), we see that there is a constant $c>0$ such that for any $f\in \dot{W}^{1,2}(X_\Gamma)$ with $f(s_0)=0$ for some $s_0\in X_\Gamma$ we have $\left\|f\right\|_{\sup}\leq c\:\mathcal{E}_{\Gamma}(f)^{1/2}$. A second fact we use is that there is a constant $c>0$ such that for any $C^2$-function $F:\mathbb{R}\to\mathbb{R}$ with bounded derivatives and any $f,g\in W^{1,2}(X_\Gamma,\mu_\Gamma)$ we have 
\begin{equation}\label{E:compest}
\left\|F(f)-F(g)\right\|_{W^{1,2}(X_\Gamma,\mu_\Gamma)}\leq c\:(\left\|F'\right\|_{\sup}+\left\|F''\right\|_{\sup})(\left\|f\right\|_{W^{1,2}(X_\Gamma,\mu_\Gamma)}+1)\left\|f-g\right\|_{W^{1,2}(X_\Gamma,\mu_\Gamma)},
\end{equation} 
as follows for instance from \cite[Proposition 3.1]{HZ12} and its proof, combined with (\ref{E:Linftybound}).
\begin{proof}
From spectral theory it is easy to see that $(e^{t\mathcal{L}_\Gamma})_{t>0}$ is contractive also on $W^{1,2}(X_\Gamma,\mu_\Gamma)$. Therefore, using the estimates $\left\|h_0\right\|_{L^2(X_\Gamma,\mu_\Gamma)}\leq \left\|h_0\right\|_{\sup}\mu(X_\Gamma)^{1/2}$ and $\left\|h_0\right\|_{\sup}\leq c\:\mathcal{E}(h_0)^{1/2}$,
\begin{multline}
\left\|u(t)\right\|_{L^2(X_\Gamma,\mu_\Gamma)}=2\mathcal{E}_\Gamma(\log w(t))^{1/2}\notag\\
\leq 2e^{\left\|h_0\right\|_{\sup}/2}\mathcal{E}_\Gamma(w(t))^{1/2}\leq 2e^{\left\|h_0\right\|_{\sup}/2}\left\|w_0\right\|_{W^{1,2}(X_\Gamma,\mu_\Gamma)}\leq c\:e^{\left\|h_0\right\|_{\sup}}\mathcal{E}_\Gamma (h_0)^{1/2},
\end{multline}
what shows (i).
To see (ii) let us write $\check{M}_0:=\max(\left\|h_0\right\|_{\sup}, \big\|\widetilde{h}_0\big\|_{\sup})$. Allowing constants to vary and using (\ref{E:compest}), 
\begin{align}
\left\|u(t)-\widetilde{u}(t)\right\|_{L^2(X_\Gamma,\mu_\Gamma)}&=2\mathcal{E}_\Gamma(\log w(t)-\log \widetilde{w}(t))^{1/2}\notag\\
&\leq c\:e^{\check{M}_0}(\mathcal{E}(w(t))^{1/2}+1)\:\mathcal{E}(w(t)-\widetilde{w}(t))^{1/2}\notag\\
&\leq c\:e^{\check{M}_0}(\mathcal{E}(w_0)^{1/2}+1)\:\mathcal{E}(w_0-\widetilde{w}_0)^{1/2}\notag\\
&\leq c\:e^{c\check{M}_0} (\mathcal{E}(h_0)^{1/2}+1)^2\:\mathcal{E}(h_0-\widetilde{h}_0)^{1/2},\notag
\end{align}
what entails (ii), note that
$\check{M}_0\leq c\:(\left\|u_0\right\|_{L^2(X_\Gamma,\mu_\Gamma)}+\big\|\widetilde{u}_0\big\|_{L^2(X_\Gamma,\mu_\Gamma)})$.
\end{proof}

\subsection{Generalized norm resolvent convergence}

The following statements are versions of results established earlier in \cite{PS17, PS17a}. We include their proofs to make the present article self-contained.

\begin{lemma}\label{L:J0lemma}
For any $f\in W^{1,2}(X_m,\mu_m)$ we have 
\[\left\|f-J_{0,m}^\ast J_{0,m} f\right\|_{L^2(X_m,\mu_m}^2\leq 54\:r^m(\max_{|w|=m}\mu(K_w))\mathcal{E}_{\Gamma_m}(f).\]
\end{lemma}

Given two edges $e,e'\in E_m$ we write $e\sim e'$ if $e\neq e'$ and $e$ and $e'$ have a common vertex.

\begin{proof}
Since $f(s)=\sum_{e\in E_m}\mathbf{1}_e(s)f(s)$ for $\mu_m$-a.e. $s\in X_m$ we have 
\[f(s)-J_{0,m}^\ast J_{0,m} f(s)=\sum_{e\in E_m}\mathbf{1}_e(s)\sum_{e'\in E_m}2^m\int_0^{2^{-m}}(f_e(s)-f_{e'}(s'))ds'\frac{\left\langle \psi_{e,m},\psi_{e',m}\right\rangle_{L^2(K,\mu)}}{\int_K\psi_{e,m}d\mu}\]
for $\mu_m$-a.e. $s$ and therefore
\begin{align}
\left\|f\right. & \left. -J_{0,m}^\ast J_{0,m}\right\|_{L^2(X_m,\mu_m)}^2\notag\\
&=\sum_{e\in E_m}2^m\int_0^{2^{-m}}\left(\sum_{e'\in E_m} 2^m\int_0^{2^{-m}}(f_e(s)-f_{e'}(s'))ds'\frac{\left\langle \psi_{e,m},\psi_{e',m}\right\rangle_{L^2(K,\mu)}}{\int_K\psi_{e,m}d\mu}\right)^2ds\left(\int_K \psi_{e,m}d\mu\right)\notag\\
&\leq \sum_{e\in E_m}\frac{2^m}{\int_K\psi_{e,m}d\mu}\int_0^{2^{-m}}\sum_{e'\in U(e)}2^m\int_0^{2^{-m}}(f_e(s)-f_{e'}(s'))^2ds'ds\sum_{\widetilde{e}\in E_m}\left\langle \psi_{\widetilde{e},m},\psi_{e,m}\right\rangle_{L^2(K,\mu)}^2,\notag
\end{align}
where $U_1(e)$ is the set of all $e'\in E_m$ such that $e'\sim e$, $U_2(e)$ is the set of all $e'\in E_m$ such that $e'\neq e$ and there exists $e''\in E_m$, such that $e''\sim e'$ and $e''\sim e$,
and
\[U(e)=\left\lbrace e\right\rbrace\cup U_1(e)\cup U_2(e).\]
Note that for $e'\in E_m\setminus U(e)$ we have $\left\langle \psi_{e',m},\psi_{e,m}\right\rangle_{L^2(K,\mu)}=0$. In case that $e'=e$ estimate (\ref{E:resistonedge}) yields
\[(f_e(s)-f_{e'}(s'))^2\leq 2^{-m}\mathcal{E}_e(f_e).\]
If $e'\in U_1(e)$ and $p$ is the common vertex of $e'$ and $e$ then, using the triangle inequality,
\[(f_e(s)-f_{e'}(s'))^2\leq  4\cdot \left\lbrace (f_e(s)-f_e(p))^2+(f_{e'}(p)-f_{e'}(s'))^2\right\rbrace \leq 4\cdot2^{-m} \left\lbrace \mathcal{E}_e(f_e)+\mathcal{E}_{e'}(f_{e'}) \right\rbrace. \]
For $e'\in U_2(e)$ with (unique) $e''$ such that $e''\sim e'$ and $e''\sim e'$ we similarly obtain 
\[(f_e(s)-f_{e'}(s'))^2\leq 9\cdot 2^{-m}\left\lbrace \mathcal{E}_e(f_e)+\mathcal{E}_{e''}(f_{e''})+\mathcal{E}_{e'}(f_{e'})\right\rbrace.\]
Inserting into the above yields
\begin{align}
\left\|f\right. & \left. -J_{0,m}^\ast J_{0,m}\right\|_{L^2(X_m,\mu_m)}^2\notag\\
&\leq 2^{-m}\sum_{e\in E_m} \frac{1}{\int_K \psi_{e,m}d\mu} \sum_{\widetilde{e}\in E_m} \left\langle \psi_{\widetilde{e},m},\psi_{e,m}\right\rangle_{L^2(K,\mu)}^2 \bigg[ \mathcal{E}_e(f_e)+4\sum_{e'\in U_1(e)} \left\lbrace \mathcal{E}_e(f_e)+\mathcal{E}_{e'}(f_{e'})\right\rbrace \notag\\
&\quad  + 9\sum_{e'\in U_2(e)} \left\lbrace \mathcal{E}_e(f_e)+\mathcal{E}_{e''}(f_{e''})+\mathcal{E}_{e'}(f_{e'})\right\rbrace \bigg],\notag
\end{align}
where in the last sum for each fixed $e$ and $e'$ the edge $e''$ is one possible connecting edge. This is less or equal 
\[36\left(\max_{e\in E_m}\frac{1}{\int_K\psi_{e,m}d\mu}\sum_{\widetilde{e}\in E_m}\left\langle \psi_{\widetilde{e},m},\psi_{e,m}\right\rangle_{L^2(K,\mu)}^2\right)r^m\mathcal{E}_{\Gamma_m}(f).\]
Using (\ref{E:sumpsiem}) it then follows that the term in brackets is bounded by $\frac32\max_{|w|=m}\mu(K_w)$.
\end{proof}

We define operators $J_{1,m}:W^{1,2}(X_m,\mu_m)\to \mathcal{D}(\mathcal{E})$ and $\widetilde{J}_{1,m}:\mathcal{D}(\mathcal{E})\to W^{1,2}(X_m,\mu_m)$ by 
\[J_{1,m} f:= H_m (f|_{V_m}), \quad f\in W^{1,2}(X_m,\mu_m),\quad\text{and}\quad \widetilde{J}_{1,m} u:=H_m(u)|_{X_m},\quad u\in\mathcal{D}(\mathcal{E}).\]

\begin{lemma}\label{L:compare01}\mbox{}
\begin{enumerate}
\item[(i)] For any $f\in W^{1,2}(X_m,\mu_m)$ we have 
\[\left\|J_{1,m} f-J_{0,m} f\right\|_{L^2(K,\mu)}\leq 36\cdot r^m (\max_{|w|=m}\mu(K_w))\:\mathcal{E}_{\Gamma_m}(f).\]
\item[(ii)] For any $u\in \mathcal{D}(\mathcal{E})$ we have 
\[\|\widetilde{J}_{1,m} u-J_{0,m}^\ast u\|_{L^2(X_m,\mu_m)}\leq \frac92 r^m(\max_{|w|=m}\mu(K_w))\:\mathcal{E}(u).\]
\end{enumerate}
\end{lemma}

\begin{proof}
To see (i) note first that 
\begin{align}
(J_{1,m}f(x)-J_{0,m}f(x))^2&=\left(\sum_{p\in V_m}c(p)\psi_{p,m}(x)\sum_{e:p\in e} (f_e(p)-\overline{f_e}) \right)^2\notag\\
&\leq 2^{-m}\left(\sum_{p\in V_m}c(p)\psi_{p,m}(x)\sum_{e:p\in e}\mathcal{E}_e(f_e)^{1/2   }\right)^2\notag
\end{align}
for any $x\in K$ by (\ref{E:resistonedge}). Consequently
\begin{align}
(J_{1,m}f(x)&-J_{0,m}f(x))^2\notag\\
 &\leq 2^{-m}\sum_{|w|=m}\int_{K_w}\left(\sum_{p\in V_m\cap K_w}c(p)\psi_{p,m}(x)\sum_{e:p\in e}\mathcal{E}_e(f_e)^{1/2}\right)^2\mu(dx)\notag\\
&\leq 2^{-m}\sum_{|w|=m}\int_{K_w}\left(\sum_{e\in E_m, e\sim K_w}\mathcal{E}_e(f_e)^{1/2}\sum_{p\in V_m\cap K_w} c(p)\psi_{p,m}(x)\right)^2\mu(dx)\notag\\
&\leq 2^{-m}\sum_{|w|=m}\left(\sum_{e\in E_m, e\sim K_w}\mathcal{E}_e(f_e)^{1/2}\right)^2\int_{K_w}\left(\sum_{p\in V_m\cap K_w}c(p)\psi_{p,m}(x)\right)^2\mu(dx)\notag\\
&\leq 36\cdot 2^{-m}\sum_{e\in E_m} \mathcal{E}_e(f_e)\:\max_{|w|=m}\mu(K_w),\notag
\end{align}
what implies (i). To show (ii) we use that by (\ref{E:sumpsiem}) we have
\[\widetilde{J}_{1,m}u(s)-J_{0,m}^\ast u(s)=\sum_{e\in E_m}\mathbf{1}_e(s)\frac{\left\langle (H_m(u)_e(s)-u,\psi_{e,m}\right\rangle_{L^2(K,\mu)}}{\int_K\psi_{e,m}d\mu}\] 
and that for fixed $e\in E_m$ and $s\in e$,
\begin{align}
\frac{\left\langle H_m(u)_e(s)-u,\psi_{e,m}\right\rangle_{L^2(K,\mu)}^2}{\int_{K}\psi_{e,m}d\mu}&\leq \int_{\supp\psi_{e,m}}(H_m(u)_e(s)-u(x))^2\psi_{e,m}(x)\mu(dx)\notag\\
&\leq \sum_{|w|=m, e\cap K_w\neq 0} r^m\mathcal{E}_{K_w}(u)\int_{\supp\psi_{e,m}}\psi_{e,m}d\mu.\notag
\end{align}
Integrating, we obtain
\begin{align}
\left\|\widetilde{J}_{1,m}u-J_{0,m}^\ast u\right\|_{L^2(X_m,\mu_m)}^2 &\leq \sum_{e\in E_m}2^m\int_0^{2^{-m}}\left\langle H_m(u)_e(s)-u,\psi_{e,m}\right\rangle_{L^2(K,\mu)}^2ds\frac{1}{\int_K\psi_{e,m}d\mu}\notag\\
&\leq \frac{1}{2}\:r^m\sum_{|w|=m} \mathcal{E}_{K_w}(u) \sum_{e\in E_m: e\cap K_w\neq \emptyset} \mu(\supp\psi_{e,m}) \notag\\
&\leq \frac92(\max_{|w|=m}\mu(K_w))\:r^m \mathcal{E}(u).\notag
\end{align}
\end{proof}

\begin{lemma}\label{L:triangleineq}
For any $u\in \mathcal{D}(\mathcal{E})$ we have 
\[\left\|u-J_{0,m} J_{0,m}^\ast u\right\|_{L^2(K,\mu)}^2\leq 6\max_{|w|=m}\mu(K_w)r^m\mathcal{E}(u).\]
\end{lemma}

\begin{proof}
We follow \cite[Lemma 2.3]{PS17} and prove that for any $u\in \mathcal{D}(\mathcal{E})$ we have 
\begin{equation}\label{E:PStrick}
\left\|u-J_{0,m}\widetilde{J}_{1,m} u\right\|_{L^2(K,\mu)}^2\leq \max_{|w|=m}\mu(K_w)r^m\mathcal{E}(u).
\end{equation}
Together with the triangle inequality, Proposition \ref{P:bounded} and Lemma \ref{L:compare01} (ii) we then obtain the result. To see (\ref{E:PStrick}) note that for any $x\in K$ we have 
\[u(x)-J_{0,m}\widetilde{J}_{1,m}u(x)=\sum_{e\in E_m}2^m \int_0^{2^{-m}} (u(x)-H_m(u)_e(s)ds\:\psi_{e,m}(x).\]
Therefore
\begin{align} 
\left\|u-J_{0,m}\widetilde{J}_{1,m}u\right\|_{L^2(X,\mu)}^2&\leq \sum_{|w|=m}\sum_{e\in E_m}\sum_{e'\in E_m}\int_{K_w}\left(2^m\int_0^{2^{-m}}(u(x)-H_m(u)_e(s))ds \right)\times\notag\\
&\quad \times \left(2^m\int_0^{2^{-m}}(u(x)-H_m(u)_e(s'))ds'\right)\:\psi_{e,m}(x)\psi_{e',m}(x)\mu(dx)\notag\\
&\leq r^m\sum_{|w|=m}\mathcal{E}_{K_w}(u)\mu(K_w)\notag
\end{align}
and (\ref{E:PStrick}) follows.
\end{proof}

\begin{lemma}\label{L:Elemma}
For any $f\in W^{1,2}(X_m,\mu_m)$ and $u\in\mathcal{D}(\mathcal{E})$ we have
\[\mathcal{E}_{\Gamma_m}(f, \widetilde{J}_{1,m}u)-\mathcal{E}(J_{1,m}f, u)=0.\]
\end{lemma}
\begin{proof}
Using the operators $H_{\Gamma_m}$ and $H_m$, 
\[\mathcal{E}_{\Gamma_m}(f,\widetilde{J}_{1,m}u)=\mathcal{E}_{\Gamma_m}(H_{\Gamma_m}f,H_{\Gamma_m}(u|_{X_m})=\mathcal{E}(H_m(f|_{V_m}), H_m(u|_{V_m}))=\mathcal{E}(J_{1,m}f,u). \]
\end{proof}

To see Theorem \ref{T:conv} it now suffices to note that by Proposition \ref{P:bounded} and Lemmas \ref{L:J0lemma}, \ref{L:compare01}, \ref{L:triangleineq} and \ref{L:Elemma} the quadratic forms $\mathcal{E}_{\Gamma_m}$ and $\mathcal{E}$ are $\delta_m$-quasi unitarily equivalent on $L^2(X_m,\mu_m)$ and $L^2(K,\mu)$ in the sense of \cite[Definition 2.1]{PS17} resp. \cite[Definition 4.4.11]{P12} with $\delta_m=54\max_{|w|=m}\mu(K_w)\:r^m$. Therefore \cite[Corollary 1.2]{PS17} implies that for any $t>0$ there exists some $C_t>0$ such that
\[\left\|e^{t\mathcal{L}}-J_{0,m} e^{t \mathcal{L}_m} J_{0,m}^\ast\right\|_{L^2(K,\mu)\to L^2(K,\mu)}\leq C_t\:\delta_m,\]
see also \cite[Theorem 4.2.10 and Proposition 4.4.15]{P12}.

\end{document}